\documentclass{amsart}
\usepackage{amsmath,amssymb}
\usepackage[all]{xy}

\theoremstyle{plain}
\newtheorem{introthm}{Theorem}
\newtheorem{theorem}{Theorem}[section]
\newtheorem{proposition}[theorem]{Proposition}
\newtheorem{lemma}[theorem]{Lemma}

\newtheorem{conjecture}[theorem]{Conjecture}

\newtheorem{theorem*}{Theorem}[]

\theoremstyle{definition}

\newtheorem{example}[theorem]{Example}

\theoremstyle{remark}
\newtheorem{remark}[theorem]{Remark}

\newcommand{\secref}[1]{Section~\ref{#1}}
\newcommand{\thmref}[1]{Theorem~\ref{#1}}
\newcommand{\propref}[1]{Proposition~\ref{#1}}

\newcommand{\exref}[1]{Example~\ref{#1}}

\def\:{{\colon}}
\def\WL{\mathrm{WL}}

\def\P{{\mathbb P}}

\def\Q{{\mathbb Q}}

\def\E{{\mathcal{E}}}
\def\D{{\mathcal{D}}}
\def\ou{{\mathrm{o/u}}}
\def\uo{{\mathrm{u/o}}}

\def\map{\mathrm{Map}}

\def\aut{\mathrm{Aut}}
\def\Baut{\mathrm{BAut}}
\def\aut{\mathrm{Aut}}

\def\der{\mathrm{Der}}
\def\Der{\mathrm{Der}}

\def\Hom{\mathrm{Hom}}
\def\Hnil{\mathrm{Hnil}}
\def\nil{\mathrm{Nil}}

\def\cat0{\mathrm{cat}_0}

\def\ker{\mathrm{ker}\,}

\def\cat{{\mathrm {cat}}}

\begin{document}

\title[Self-Equivalences of a Fibration]
{The Rational Homotopy Type of the Space of Self-Equivalences of a
Fibration}

\author{Yves F{\'e}lix}
\address{Institut Math{\'e}matique,
Universit{\'e} Catholique de Louvain, B-1348 Louvain-la-Neuve
Belgique}
\email{felix@math.ucl.ac.be}

\author{Gregory Lupton}
\address{Department of Mathematics,
      Cleveland State University,
      Cleveland OH 44115 U.S.A.}
\email{G.Lupton@csuohio.edu}

\author{Samuel B. Smith}
\address{Department of Mathematics,
  Saint Joseph's University,
  Philadelphia, PA 19131 U.S.A.}
\email{smith@sju.edu}

\date{\today}

\keywords{fibre-homotopy equivalences, Samelson Lie algebra, function space,
Sullivan minimal model, derivation}

\subjclass[2000]{55P62, 55Q15}

\begin{abstract}
Let $\aut(p)$ denote
the  space   of  all self-fibre-homotopy equivalences of a fibration $p \colon E \to B$.   When $E$ and $B$ are simply connected CW complexes with $E$ finite, we identify  the rational Samelson Lie algebra of this
monoid by means of an isomorphism:
$$\pi_*(\aut(p)) \otimes \Q \cong   H_*(\der_{\land V}(\land V\otimes \land W)).
$$
Here $\land V\to \land V \otimes  \land W$ is the  Koszul-Sullivan
model of the fibration  and $\der_{\land V}(\land V\otimes \land W)$ is the DG Lie algebra  of
derivations  vanishing on $\land V.$ We   obtain    related identifications  of the rationalized homotopy groups of fibrewise mapping spaces and of the rationalization of the  nilpotent
group  $ \pi_0(\aut_\sharp(p))$ where $\aut_\sharp(p)$ is a
fibrewise adaptation of the submonoid of maps inducing the  identity on homotopy groups.
\end{abstract}

\maketitle

\section{Introduction}\label{sec:intro}
Given  a fibration  $p \colon E \to  B$
of connected CW complexes, let  $\aut(p)$ denote  the space  of unpointed fibre-homotopy self-equivalences   $f \: E\to
E$ topologized as a subspace of    $\map(E, E).$ By a theorem of Dold \cite[Th.6.3]{Dold},  $\aut(p)$ corresponds to the space of ordinary homotopy self-equivalences $f \colon E \to E$ satisfying $p \circ f=p,$
The space $\aut(p)$ is  a monoid with multiplication given by composition of maps.  In general, $\aut(p)$
is a disconnected space with possibly infinitely many components.
The group of components $\pi_0(\aut(p)) $ is the group of fibre-homotopy equivalence classes of  self-fibre-homotopy equivalences of
$p.$  We denote this group by $\mathcal{E}(p)$. We make a general study of $\aut(p)$, especially in rational homotopy theory.

The monoid $\aut(p)$  appears as an
object of interest in many different situations.  When $B$ is a
point, $\aut(p) \simeq \aut(E)$ is the monoid of (free) self-homotopy
equivalences of the space $E$, and $\E(p) = \E(E)$, the group of
self-equivalences of $E$. Taking  $p \colon   PB \to B$ to be the path-space
fibration, we have  $\aut(p)
 \simeq \Omega B$  and $\E(p) = \pi_1(B)$  (see \exref{ex:Loops B as autxi}, below).
When   $p$ is a covering map, $\aut(p)$
contains the group of deck transformations of the covering.

As regards the  homotopy type of $\aut(p)$,
under reasonable hypotheses on $p$, the monoid $\aut(p)$   is   a   grouplike space  of CW type (see Proposition \ref{pro:CW2}).
The path components of a well-pointed grouplike space are all of the same homotopy type. Thus we focus on the path component   of  the identity  map which we denote $\aut(p)_\circ$.
Our first main result gives a complete description of the rational H-homotopy type of this connected grouplike space when $E$ is finite.

Before describing our main results, we review some background and notation.
 Recall the homotopy groups $\pi_*(G)$ of a connected, grouplike space $G$ admit a   natural
bilinear pairing $[ \, , \, ]$ called the {\em Samelson product}  (see \cite[Ch.III]{GW}).
If $G$ has multiplication $\mu$  the pair $(G, \mu)$   admits a rationalization as in \cite{HMR} yielding an  H-space $(G_\Q, \mu_\Q)$ which is unique up to H-equivalence.  We refer to the
H-homotopy type of the pair $(G_\Q, \mu_\Q)$ as  the {\em rational H-type} of $(G, \mu).$  If $(G, \mu)$ is grouplike then
$(G_\Q, \mu_\Q)$ is also.  In this case, the rational H-type of  $G$
is completely determined by the rational Samelson algebra $  \pi_*(G) \otimes \Q, [ \, , \, ].$
Specifically, two grouplike spaces are rationally H-equivalent if and only if they have
isomorphic Samelson Lie algebras \cite[Cor.1]{Sch}. We  identify the rational Samelson Lie algebra of $\aut(p)_\circ$ in the context of Sullivan's rational homotopy theory. Our    reference for  rational homotopy theory is  \cite{FHT}.

 In  \cite{IHES},  Sullivan defined a functor $A_{PL}(-)$ from topological spaces to
commutative differential graded algebras over $\Q$ (DG algebras for
short). The functor  $A_{PL}(-)$   is connected to the cochain algebra functor $C^*(-;\Q)$
by a sequence of natural quasi-isomorphisms.
Let $\land V$   denote the free commutative graded algebra on
the graded rational  vector space $V.$
A DG algebra $(A,d)$ is   a
\emph{Sullivan algebra} if $A\cong \land V$ and if $V$ admits a
basis $(v_i)$ indexed by a well ordered set such that $d(v_i)\in
\land (v_j, j<i)$.  If the differential $d$ has image in the decomposables
of $\land V$ we say $(A, d)$ is {\em minimal}. Filtering by product length,
 a minimal DG algebra is seen to be a Sullivan algebra.
A DG algebra $(A, d)$ is a {\em Sullivan model} for $X$ if
$(A, d)$ is a Sullivan algebra and there is a quasi-isomorphism
$(A, d) \to A_{PL}(X).$ If $(A, d)$ is minimal it is then the {\em Sullivan minimal model}  of $X$.

A fibration $p \: E \to B$  of simply connected CW complexes admits a
  \emph{relative minimal model} (see \cite[Pro.15.6]{FHT}). This is an injection of DG
algebras $I \: (\land V,d)\to (\land V\otimes W,D)$ equipped with
quasi-isomorphisms $\eta_B$ and $\eta_E$ which make the following
diagram commutative:
\begin{equation} \label{eq:rel model} \xymatrix{
(\land V,d) \ar[d]_{\eta_B}^{\simeq} \ar[r]^-{I} & (\land
V\otimes\land W,D) \ar[d]_{\eta_E}^{\simeq} \\
A_{
PL}(B) \ar[r]^-{A_{PL}(p)} & A_{PL}(E)
}\end{equation}
Here $(\land V,d)$ is the Sullivan minimal model of $B$  while $(\land
V\otimes\land W,D)$ is a Sullivan (but generally non-minimal) model of
$E$. The differential $D$ satisfies    $D(W)\subset (\land^+V\otimes\land W) \oplus (\land
V\otimes\land^{\geq 2}(W))$ and further $W$ admits a basis $w_i$
indexed by a well ordered set such that $D(w_i)\in \land
V\otimes\land (w_j, j<i)$.

A derivation $\theta$ of degree $n$ of a DG algebra $(A, d)$
will mean a linear map {\em lowering} degrees by $n$ and satisfying $
\theta(ab)= \theta(a)b - (-1)^{n|a|}a\theta(b)$ for $a,b \in
A.$ We write $\der^n(A)$ for the vector  space of all degree $n$
derivations.  The graded vector space $\der^*(A)$ has the structure
of a DG Lie algebra with the commutator bracket $[\theta_1, \theta_2]
= \theta_1 \circ \theta_2 - (-1)^{|\theta_1||\theta_2|}\theta_2
\circ \theta_1$ and differential $\D(\theta) = [d, \theta].$
Given a DG subalgebra $B \subseteq A$ we write $\der^*_B(A)$ for the space
of derivations of $A$ that vanish on $B$.  The bracket and differential evidently restrict to
give a DG Lie algebra $(\Der^*_B(A), \D)$.
More generally, given a DG algebra map $\phi \colon A \to A'$ we write
$\der^n(A, A'; \phi)$ for the space of degree $n$ linear maps satisfying
$\theta(ab)= \theta(a)\phi(b) + (-1)^{n|a|}\phi(a)\theta(b)$
with differential $\D (\theta) = d_{A'} \circ \theta - (-1)^{|\theta|}\theta \circ d_A.$
The pair $(\der^*(A, A'; \phi),  \D )$ is then a DG vector space.  Given a DG subalgebra $B \subseteq A$ we have the DG vector subspace $(\der_B^*(A, A'; \phi), \D)$ of derivations  that vanish on $B$.

We now describe our main results.  Given a fibration $p \colon E \to B$ of simply connected CW complexes,   choose and
fix a relative minimal model $I \colon (\land
V,d) \to (\land V\otimes \land W,D)$ as above.
We have:

\begin{introthm}\label{thm:main}%
Let  $p \colon   E \to B$ be a fibration of
simply connected CW complexes with  $E$   finite.
There is an isomorphism of graded Lie algebras
$$\pi_*(\aut(p)_\circ)\otimes\Q\, \cong \, H_*(\der_{\land V}(\land V\otimes \land W))
.$$
\end{introthm}
By the remarks above, Theorem \ref{thm:main} completely determines the rational H-type of $\aut(p).$
Taking $B = *$, we recover
Sullivan's Lie algebra isomorphism
$$\pi_*(\aut(E)) \otimes \Q \cong H_*(\der( \land W))$$
described in \cite[Sec.11]{IHES}. Here  $(\land W, d)$ is the minimal
model for $E$.

 We obtain Theorem \ref{thm:main} as  a sharpened special case of a general calculation of the rational homotopy groups of a fibrewise mapping space.  Let $p' \colon E' \to B'$ be a second fibration  and $f \colon E' \to E$ be a fibrewise  map covering a map $g \colon B' \to B$.
Let  $$\map_g(E', E) = \{ h \colon E' \to E  \mid  p \circ h =  g \circ p' \}
$$
denote the  function space of maps over $g$ topologized as a  subspace of $\map(E', E)$.    Write $\map_g(E', E; f)$ for the path component of $f$.
We   compute the rational homotopy groups of this  space as    an extension of Sullivan's original approach. Fix a relative minimal model $I' \colon (\land V', d) \to (\land V' \otimes \land W', D)$ for $p' \colon E' \to B'$ as above.  The map $f \colon E' \to E$ has a model $\mathcal{A}_f \colon \land V \otimes \land W \to \land V' \otimes \land W'$.  A homotopy class $\alpha \colon S^n \to \map_g(E', E; f)$ has adjoint
$F \colon  E' \times S^n \to E$ satisfying $p \circ F(x, s) = g \circ p' (x)$ and $F(x,*) = f(x)$
for all $(x, s) \in  E' \times S^n .$  We show $F$ has    a  DG model
$$ \mathcal{A}_F \colon \land V  \otimes \land W \to \land V' \otimes \land W'  \otimes  \left( \land(u)/\langle u^2 \rangle \right)  $$
satisfying $\mathcal{A}_F(\chi) =   \mathcal{A}_f(\chi) + u  \theta(\chi)
$
for  $ \chi \in \land V \otimes \land W.$   Here $|u| = n$ so that $(\land (u) /\langle u^2 \rangle, 0)$ is a Sullivan model for $S^n$. The degree $n$ map  $\theta$ is then  an  $\mathcal{A}_f$-derivation as defined above and  we show  $\theta$    vanishes on $\land V$, that is, $\theta$  is an element of   $ \der^n_{\land V}(\land V\otimes \land W, \land V' \otimes \land W'; \mathcal{A}_f)$.    We prove the assignment $\alpha \mapsto \theta$ gives rise to a rational isomorphism:

\begin{introthm}\label{thm:main2}%
Let $p' \colon E' \to B'$ and $p \colon   E \to B$  be   fibrations of
simply connected CW complexes with  $E'$   finite. Let  $f \colon E' \to E$ be a fibrewise map  with  model $\mathcal{A}_f$ as above.
There is  an isomorphism of vector spaces
$$\pi_n(\map_g(E', E; f))\otimes\Q\, \cong \, H_n(\der_{\land V}(\land V\otimes \land W, \land V' \otimes \land W'; \mathcal{A}_f))
.$$
for $ n\geq 2.$  Further, the group $\pi_1(\map_g(E', E; f))$ is nilpotent and satisfies
$$\mathrm{rank}(\pi_1(\map_g(E', E; f))) = \mathrm{dim}_\Q\left(  H_1(\der_{\land V}(\land V\otimes \land W, \land V' \otimes \land W'; \mathcal{A}_f)) \right).$$
\end{introthm}
Again, taking $B' = B = *$,  we recover  known results;  in this case the vector space isomorphism
$$ \pi_n(\map(X, Y;f)) \otimes \Q \cong H_n(\der(\mathcal{M}_Y, \mathcal{M}_X; \mathcal{M}_f))$$
for $n \geq 2$ appears in \cite{BL, LS,  BM} where here $\mathcal{M}_f \colon \mathcal{M}_Y \to \mathcal{M}_X$ is the Sullivan minimal model of $f \colon X \to Y.$  The result on
$\mathrm{rank}(\pi_1(\map(X, Y;f))) $ is
\cite[Th.1]{LS2}.
Our proof of Theorem \ref{thm:main2} is  an adaptation of the proofs of  \cite[Th.2.1]{LS}  and \cite[Th.1]{LS2}
with adjustments made for the fibrewise setting.

Finally,  as regards the group $\E(p)$ of path components of $\aut(p)$ we remark that
this group is not, generally, a nilpotent group and so not directly amenable to rational
homotopy theory. In Section \ref{sec:components}, we consider the  subgroup $\E_\sharp(p)$ of $\E(p)$  consisting of homotopy equivalence classes of maps  $f \colon E \to E$ over
$B$ inducing the identity on image and cokernel of the connecting homomorphism in the long exact sequence on homotopy groups of the fibration. The group $\E_\sharp(p)$   is nilpotent and localizes well for $E$ finite by the fibrewise extension of results of    Maruyama \cite{Maru}. Let $D_0 \colon W  \to V  $
be the linear part of $D$ in the Sullivan model for $p.$  Write $W = W_0 \oplus W_1$ where $W_0 = \ker D_0$ and $W_1$ is a vector space complement. Denote by $\der^0_{\#}(\land V\otimes \land W)$ the vector space of
derivations $\theta$ of degree $0$ of $\land V\otimes\land W$ that
satisfy $\theta (V)=0$, $\D(\theta) = 0 $, $\theta (W_0)\subset V
\oplus \land^{\geq 2}(V\oplus W)$, and $\theta (W_1)\subset V \oplus
W_0 \oplus \land^{\geq 2}(V\oplus W)$.Ê
   Define   $$H_0(\der_\sharp(\land V\otimes\land W)) = \mbox{coker}\, \{
\D : \mbox{Der}^1_{\land V}(\land V\otimes\land W)\to
  \mbox{Der}^0_{\#}(\land V\otimes\land W) \} .$$
The case $B = *$  in the following result   is  originally due to Sullivan \cite[Sec.11]{IHES}.
See also \cite[Th.12]{Sal}.

\begin{introthm} \label{thm:main3} Let $p \colon E \to B$ be a fibration of simply connected CW
complexes with $E$ finite. There is a group isomorphism $$  \mathcal{E}_{\#}(p)_\Q \cong H_0(\der_\sharp(\land V\otimes\land W))
.$$
\end{introthm}

The paper is organized as follows.  In Section \ref{sec:prelims}, we
  prove several general results
concerning $\aut(p)$.  We  prove Theorem \ref{thm:main2} in Section \ref{sec:fibwise} as a consequence of some  results in fibrewise DG  homotopy theory.  We prove     Theorem \ref{thm:main} in Section  \ref{sec:main} and
   deduce some consequences     in Section \ref{sec:examples}. Finally,  in
\secref{sec:components} we turn to the group of components $\E(p)$
of $\aut(p)$ and prove   Theorem \ref{thm:main3}.      We conclude
with  an example to show that $\E_{\#}(p)$ is not generally a
subgroup of $\E_{\#}(E).$

\section{Basic Homotopy theory of $\aut(p)$}%
\label{sec:prelims}
Function spaces  are not generally  suitable for and well-behaved under localization.
 In this section,  we  show that, under reasonable hypotheses, the fibrewise function spaces are nilpotent CW complexes admitting natural localizations. This result is a direct consequence of the corresponding,  standard results on function spaces.  We then give a variety of  examples concerning the monoid $\aut(p)$.

Start with  a commutative  diagram  of connected CW complexes
\begin{equation} \label{eq:fibsquare}
\xymatrix{E' \ar[d]_{p'} \ar[r]^f & E \ar[d]^{p} \\ B' \ar[r]^{g} & B}\end{equation}
 with vertical maps fibrations.  Composition with $p$ gives a  fibration of ordinary function spaces
$$ p_* \colon \map(E', E;  f) \to \map(E', B; p \circ f).$$
The fibre of $p_*$ over the basepoint $p \circ f$ is the space $$\mathcal{F} = \{ h \colon E' \to E \mid h \simeq f \hbox{\, and \, } p \circ h = p \circ f  \}$$
The  path component of $f$ in  $\mathcal{F}$ is just the fibrewise function space   $\map_g(E', E; f),$ introduced above.

Following Hilton-Mislin-Roitberg \cite{HMR},     say  a space $X$ is a  {\em nilpotent} space if $X$ is a connected CW complex with $\pi_1(X)$ nilpotent and acting nilpotently on the higher homotopy groups of $X$.
As shown in \cite[Ch.II]{HMR}, the class of nilpotent spaces   admit $\P$-localizations for $\P$ any  fixed collection of primes. That is, there is a space $X_\P$ and a map  $\ell_X \colon X \to X_\P$ such that
 $X_\P$ is a  $\P$-local space and    $\ell_X$ is
a $\P$-local-equivalence. Given a map $f \colon X \to Y$  with $Y$ nilpotent, we write $f_\P  = \ell_Y \circ f \colon X \to Y_\P.$
 We will mostly be interested in the case $\P$ is empty  in which case we write $f_\Q \colon X \to Y_\Q.$

The assignment $X \mapsto X_\P$ is functorial at the level of  homotopy classes of maps \cite[Prop.II.3.4]{HMR}.  We observe a refinement of this fact for fibrations.  Suppose  $p \colon E \to B$ is a fibration of nilpotent spaces.  We may construct $\P$-localizations for $E$ and $B$ creating   a commutative diagram:  \begin{equation} \label{eq:fiblocalization}  \xymatrix{E \ar[r]^{\ell_E} \ar[d]_p &E_\P \ar[d]^{q}  \\
B \ar[r]^{\ell_B} & B_\P} \end{equation}
where $q$ is a fibration.  To see this, first  construct
  $\P$-localizations  $\ell'_E \colon E \to  E'_\P$ and $\ell_B \colon B \to B_\P$
as relative CW complexes.
 Then  $  p \circ \ell_B \colon E \to B_\P$   extends to a map $q' \colon E'_\P \to B_\P$.  Transform $q'$  into a fibration by injecting $q'$ into the associated homotopy fibration yielding $q \colon E_\P \to B_\P$ .  Let $\ell_E \colon E \to E_\P$ denote the composition of the equivalence $E'_\P \simeq E_\P$ with $\ell'_E.$

Now suppose given a commutative diagram (\ref{eq:fibsquare}) with all spaces nilpotent.
We then have  a commutative square:
 $$  \xymatrix{E' \ar[d]_{p'} \ar[r]^{f_\P} & E_\P \ar[d]^{p} \\ B' \ar[r]^{g_\P} & B_\P} $$
 and a map $$(\ell_E)_*  \colon \map_g(E', E; f) \to \map_{g_\P}(E', E_\P; f_\P)$$
 induced by composition with $\ell_E.$
 We prove

 \begin{proposition} \label{pro:CW}    In the commutative square (\ref{eq:fibsquare}), suppose $E'$ is a finite CW complex and $E$ and $B$ are nilpotent spaces.
 Then  $\map_g(E', E;f )$  is a  nilpotent space and composition with $\ell_E$ induces a $\P$-localization map
$$ (\ell_{E})_* \colon \map_g(E', E; f) \to \map_{g_\P}(E', E_\P; f_\P).$$
 \end{proposition}
\begin{proof}
By \cite{Mil},   the function  spaces $\map(E', E; f)$ and
$\map(E', B; p \circ f)$ are of CW type. Thus  the fibre of $p_* \colon \map(E', E; f) \to \map(E', B; p \circ f)$ is CW as well by \cite[Lem.2.4]{PJK}.  Further,  by  \cite[Cor.II.2.6 and Th.II.3.11]{HMR}, $\map(E', E; f)$ and
$\map(E', B; p \circ f)$ are nilpotent spaces with rationalizations induced by composition with $\ell_E$ and $\ell_B$.    By \cite[Th.II.2.2]{HMR},   each component of the fibre of $p_*$ is nilpotent. (As remarked at the end of the proof \cite[p.63]{HMR} the result holds for non-connected fibres.)  Thus $\map_g(E', E; f)$ is a nilpotent space.  Finally,  we see $\ell_E$ induces a $\P$-localization by  the Five Lemma.    \end{proof}

As regards the monoid $\aut(p)_\circ$,  we  may sharpen the first part of this  result:

\begin{proposition} \label{pro:CW2}   Let $p \colon E \to B$ be a
fibration of connected CW complexes with connected fibre $F$. If  either $E$ or  $B$ is  finite
then $\aut(p)$ has the homotopy type of a CW complex and the
H-homotopy type of a loop-space.
\end{proposition}
\begin{proof} For the  CW structure in the case $B$ is finite,  we use the identity
$\aut(p)\simeq \Omega \map(B, BF; h)$ where $BF$ is the base of the universal fibre with fibre $F = p^{-1}(b_0)$ and $h \colon B \to BF$ is the classifying map (see \cite[Th.1]{Got} and \cite[Th.3.3]{BHMP}).
Since $BF$ is CW in this case, applying  \cite{Mil} again gives the result. In either case,
$\aut(p)$ is a strictly associative CW monoid and so admits a
Dold-Lashof classifying space $\Baut(p)$.  Thus $\aut(p) \simeq
\Omega \Baut(p)$ \cite[Satz.7.3]{Fuchs}\end{proof}

In our main results,  we consider $\aut(p)$ for  $E$ finite. By Proposition \ref{pro:CW2} this restriction is not necessary for nilpotence  since a connected CW monoid is automatically nilpotent. However,  we will make use of   the second   statement in   Proposition  \ref{pro:CW} in the proof of Theorems \ref{thm:main} and \ref{thm:main2}.

We next recall an interesting invariant of     a connected grouplike space $G$,  the {\em
homotopical nilpotency} of $G$  as studied by Berstein and
Ganea~\cite{B-G}. It is   defined as follows:  Using the homotopy inverse,  we have  commutator maps   $\varphi_n \colon G^n \to G$:
Here $\varphi_1$ is the identity, $\varphi_2(g, h) = ghg^{-1}h^{-1}$ is the usual commutator, and $\varphi_n = \varphi_2 \circ (\varphi_{n-1} \times \varphi_1).$
The homotopical nilpotency  $\Hnil (G)$  is then the least integer $n$ such
that   $\varphi_{n + 1}$   is null-homotopic.
 The
{\em rational homotopical nilpotency} $\Hnil_\Q(G)$ of $G = (G, \mu)$ is defined to be the
homotopical nipotency of $G_\Q = (G_\Q, \mu_\Q)$.  The inequality
$\Hnil_\Q(G) \leq \Hnil(G)$ is direct from definitions. When $\Hnil (G) = 1$  (respectively, $\Hnil_\Q(G) = 1$) we say  G is   {\emph{homotopy abelian}} (respectively, {\em rational homotopy abelian}).

When $G = \Omega X$  is a loop-space,   $\Hnil(G)$ is directly related to the nilpotency $\nil(\pi_*(G))$ of the
Samelson bracket $[ \, , \, ]$ on $\pi_*(G)$ and to the length  $\WL(X)$ of the longest Whitehead bracket in $\pi_*(X).$ If $X$ is a nilpotent space, write $\WL_\Q(X)$ for the Whitehead
length of  the rationalization of $X_\Q$ of $X$.

\begin{proposition} \label{pro:nil and Hnil}
Let  $G$  be a connected  CW loop space,  $G \simeq \Omega X$ for some simply
connected space $X$.   Then
$$  \Hnil_\Q(G)= \nil(\pi_*(G)\otimes \Q) =
\WL_\Q(X) \leq \WL(X) = \nil(\pi_*(G)) \leq  \Hnil(G).$$
\end{proposition}

\begin{proof}
The equality $\WL(X) = \nil(\pi_*(G))$ (and its rationalization)  is
a  consequence of the identification of the Whitehead product with
the Samelson product via the isomorphism $\pi_*(\Omega X) \cong
\pi_{*+1}(X)$ \cite[Th.X.7.10] {GW}. The   inequality
$\nil(\pi_*(G)) \leq \Hnil(G)$ is   \cite[Th.4.6]{B-G}. The
equality $\nil(\pi_*(G)\otimes \Q) = \WL_\Q(X)$ is \cite[Lem.4.2]{ArkCur} (see also \cite[Th.3]{Sal}). Finally, the
inequality $\WL(X_\Q) \leq \WL(X)$ is immediate from the definition.
\end{proof}

We give some examples and direct calculations regarding  $\aut(p)$.

\begin{example}\label{ex:Loops B as autxi}%
Let $PB$ be  the space of Moore
paths  at $b_0 \in B$:
$$PB = \{ (\omega, r) \,\vert \, r\geq 0, \omega : \mathbb (R_{\geq 0},0) \to
(B, b_0) \,, \mbox{such that $\omega (s)=\omega (r)$ for $r\geq
s$}\,\}\,.$$
 Let $p \colon PB \to B$ given by $p(\omega, r)=\omega
(r)$ be the path-space fibration.
Then there is an H-equivalence $$
  \aut(p) \simeq \Omega B$$
where  $\Omega B = \{ (\sigma, s) \in PB  \mid \sigma(s) = b_0 \}$ is  the
space of Moore loops at $b_0$.  To see this, define $\theta \: \Omega B \to \aut(p)$   given  by
$\theta ((\sigma, s))(\omega, r) = (\sigma * \omega, s + r)$ and $\psi \colon \aut(p) \to \Omega B.$
Given $f \in \aut(p)$, $f$ restricts to an equivalence $f \colon \Omega B \to \Omega B.$
Define $\psi(f)=f(c_{b_0}, 0)$ where $c_{b_0}$ is the   constant loop.   Clearly
$\psi\circ \theta = id$. On the other hand, a homotopy $H \: \aut(p) \times [0,1] \to \aut(p),$ between
$\theta\circ \psi $ and the identity is defined by
$$H(f,t)(\omega, r) = f(\omega_{tr}, tr)*(\omega^b_tr, r-tr)\,.$$ Here $\omega^b_s(t) = \omega (t+s)$ and $$\omega_{s} (t)=
\left\{\begin{array}{ll} \omega
(t) & t\leq s\\
\omega (s)& t\geq s.
\end{array}
\right. $$
Using Proposition \ref{pro:nil and Hnil}, we conclude  $$\Hnil(\aut(p)_\circ) \geq \WL (\widetilde{B} ) \hbox{\, and \, }  \Hnil_\Q(\aut(p)_\circ) = \WL_\Q (\widetilde{B} )$$
  where
$\widetilde{B}$ denotes the universal cover of $B$.
 \end{example}


%

\begin{example}  \label{ex:Halperin} Observe that when $\pi \colon  F \times B
\to  B$ is the (trivial) product fibration  then by adjointness we have $$\aut(\pi) \cong \map(B, \aut(F))$$ with
pointwise multiplication in the latter space.   More generally,   by the
fibre-homotopy invariance of $\aut(p)$ we have this identification for any fibre-homotopy trivial  fibration $p \colon E \to B$ with fibre $F = p^{-1}(b_0)$.   By \cite[Th.4.10]{LPSS}, we have
$$\Hnil(\aut(p)_\circ) = \Hnil(\aut(F)_\circ)$$
in this case.
\end{example}

\begin{example} We next consider a class of fibrations of particular interest in rational homotopy theory. An {\em $F_0$-space} $F$ is   a simply connected  complex with  $H^*(F; \Q)$ and  $\pi_*(F) \otimes \Q$ both  finite-dimensional and $H^{odd}(F; \Q) = 0$. Examples include (finite products of) even-dimensional spheres, complex projective spaces and homogeneous spaces $G/H$ of equal rank compact pairs.  Regarding these spaces, we have the following famous conjecture: \begin{conjecture}  {\em (Halperin \cite{Hal})} \label{HC}  The rational Serre spectral sequence collapses at the $E_2$-term for every fibration $p \colon E \to B$ of simply connected CW complexes with fibre $F$ an $F_0$-space.
\end{conjecture}
The conjecture has been affirmed in many cases, including the examples mentioned.  
A connection to  spaces of  self-equivalences is provided by  a   theorem due to  Thomas and Meier:
An $F_0$-space $F$
  satisfies the Conjecture  \ref{HC}  if and only if $\aut(F)_\circ$ has vanishing even degree rational homotopy groups (see  \cite[Th.A]{Me}). Thus $\nil(\pi_*(\aut(F)_\circ) \otimes \Q) = 1$ in this case for degree reasons and so   $ \Hnil_\Q(\aut(F)_\circ) =  1.$
  By Example \ref{ex:Halperin}, we conclude that $\aut(p)_\circ$ is rationally homotopy
  abelian for any trivial fibration $p \colon E \to B$ of simply connected spaces
  with fibre $F$ an $F_0$-space satisfying Conjecture  \ref{HC}.

This result extends  to a nontrivial  fibrations via the identity
 $$\aut(p)_\circ \simeq \Omega_\circ\map(B, \Baut_1(F); h).$$
 Here we write $\Baut_1(F) = \Baut(F)_\circ$
 for  the classifying space of the connected monoid $\aut(F)_\circ.$
 By  \cite[Th.A]{Me}, $\aut(F)_\circ$ has oddly graded rational homotopy groups.
 Thus $\Baut_1(F) $ has evenly graded rational homotopy groups and so
 $\Baut_1(F) $ is a rational grouplike space. We thus have an equivalence
 $$\map(B, \Baut_1(F)_\Q; h_\Q) \simeq \map(B, \Baut_1(F)_\Q; 0)$$ which loops to
 an H-equivalence
$$ \Omega_\circ\map(B, \Baut_1(F)_\Q; h_\Q) \simeq \Omega_\circ\map(B, \Baut_1(F)_\Q; 0).$$ Combining, we see
$$\begin{array}{lll} (\aut(p)_\circ)_\Q &\simeq  \Omega_\circ\map(B, \Baut_1(F)_\Q; 0) & \simeq
\map(B, \Omega_\circ\Baut_1(F)_\Q; 0) \\
& \simeq \map(B, (\aut(F)_\circ)_\Q; 0). \end{array}$$
Thus,   since $\aut(F)_\circ$ is rationally homotopy abelian so is $\aut(p)_\circ$   as in Example \ref{ex:Halperin}.
 \end{example}

\begin{example} \label{ex:principal} Let $p \colon  E\to B$ be a principal $G$-bundle.
Multiplication by an element of $G$ induces a morphism of $H$-spaces
$G \to \aut(p)$ while  evaluation at the identity gives a left inverse
 $\aut(p) \to G$. Thus   $G$ is a retract of
$\aut(p)$. It follows easily that
$$\Hnil(\aut(p)_\circ) \geq \Hnil(G_\circ).$$
\end{example}

\section{Rational homotopy groups of fibrewise mapping spaces} \label{sec:fibwise}

We now  consider the diagram of fibrations (\ref{eq:fibsquare})
above defining the fibrewise mapping space $\map_g(E', E; f)$.  We
assume all the spaces $E', E, B', B$ are simply connected CW
complexes.  Our calculation of $\pi_n(\map_g(E', E; f)) \otimes \Q$
will follow the line of proof of  \cite[Th.2.1]{LS}. Namely, we will
define a homomorphism $\Phi'$ from the ordinary homotopy group of
the function space to the DG vector space  of derivations and prove
$\Phi'$  induces an isomorphism $\Phi$ after rationalization.

As in \cite{LS}, the construction of $\Phi'$ depends on some   DG
algebra  homotopy theory. Let  $\alpha \in \pi_n(\map_g(E', E; f))$
be represented by $a \colon S^n \to \map_g(E', E; f)$ with adjoint
$F \colon E' \times S^n \to E.$ Write $i \colon E' \to    E' \times
S^n$ for the based inclusion and $q \colon  E' \times S^n \to B'$
for $p'$ composed with the projection, so that $q\circ i = p'$. The
map $F$ is then a map ``under $f$'' and ``over $g$". That is, we
have $F \circ i = f$ and $ p\circ F = g \circ q.$ To define
$\Psi'(\alpha),$ we show the Sullivan model of such a map $F$ can be
taken  as a DGA map ``under" a model for $g$ and ``over" one for
$f$.   We prove, in fact, that this assignment sets up a bijection
between homotopy classes of maps when $p \colon E \to B$ is replaced
by its rationalization.  Precise definitions and statements  follow.

Let $Z$ be any nilpotent space and consider the diagram
\begin{equation}\label{eq:over/under}  \xymatrix{E' \ar[d]_{i} \ar[rrd]^{f} \\
E' \times Z  \ar[d]_{q} \ar@{.>}[rr]^{F} && E \ar[d]^{p} \\
B' \ar[rr]^{g} && B }\end{equation}
If $F$ makes the diagram strictly commute we say $F$ is a map {\em
over  $g$} and {\em under $f$}. We say two such maps $F_0$ and $F_1$
are {\em homotopic over $g$ and under $f$} if there is a homotopy
$H$ from $F_0$ to $F_1$  through maps over $g$ and under $f$, i.e.,
if $H$ makes the diagram
\begin{equation}\label{eq:over/under homotopy}
\xymatrix{E'\times I \ar[d]_{i\times 1} \ar[rrd]^{T_f} \\
E' \times Z \times I \ar[d]_{q\times 1} \ar@{.>}[rr]_-{H} && E \ar[d]^{p} \\
B'\times I \ar[rr]^-{T_g} && B }\end{equation}
commute, where $T_f$ and $T_g$ are stationary homotopies at $f$ and
$g$, respectively.   Write
$$[E' \times Z, E]_{\ou}$$ for the set of homotopy classes of maps
over $g$ and under $f$.

On the DG algebra side,
let $\eta_Z \colon (C, d) \to (A(Z), \delta_Z)$ be a Sullivan model for $Z$.
We then have a corresponding diagram:
\begin{equation}\label{eq:o/u}
\xymatrix{ \land V \ar[rr]^{\mathcal{M}_g} \ar[d]_{I} && \land V' \ar[d]^{  I''} \\
\land V \otimes \land W \ar@{.>}[rr]^{\Gamma}\ar[rrd]_{\mathcal{A}_f} &&  \land V' \otimes \land W' \otimes C
\ar[d]^{P} \\
&& \land V' \otimes \land W'. }
\end{equation}
Here, $I\colon \land V \to \land V \otimes \land W$ is the inclusion
$I(\chi) = \chi\otimes1$, for $\chi \in \land V$, $I''$ is the
inclusion $I''(\chi') = \chi'\otimes1\otimes1$, for $\chi' \in \land
V'$, and $P=(1\otimes1)\cdot\varepsilon$ is the obvious projection
so that the composition $P\circ I''$ gives the inclusion $I' \colon
\land V' \to \land V' \otimes \land W'$, with $I'(\chi') =
\chi'\otimes1$, for $\chi' \in \lambda V'$. The inclusion $J \colon
\land V' \otimes \land W' \to \land V' \otimes \land W' \otimes C$
splits $P$.  A DG algebra map $\Gamma$ making the diagram
(\ref{eq:o/u}) strictly commute will be called a map {\em under
$\mathcal{M}_g$} and {\em over $\mathcal{A}_f$}. Recall that a DG
homotopy between $\Gamma_0$ and $\Gamma_1$ may be taken as a map
$$\mathcal{H} \colon \land V \otimes \land W \to \land V' \otimes
\land W' \otimes C \otimes \land (t, dt)$$ $\pi_i \circ\mathcal{H} =
\Gamma_i$ for $i = 0, 1.$ Here   $\land(t, dt)$ is the contractible
DG algebra with $|t| =0$; the maps $\pi_i \colon  \land V' \otimes
\land W' \otimes C \otimes \land (t, dt) \to  \land V' \otimes \land
W' \otimes C$
 correspond to sending $t \mapsto i$ and $dt \mapsto 0$   \cite[\S 12.b]{FHT}.
We say $\Gamma_0$ and $\Gamma_1$ are {\em homotopic   under
$\mathcal{M}_g$ and over $\mathcal{A}_f$} if $\mathcal{H}$ is a DG
homotopy through maps under $\mathcal{M}_g$ and over
$\mathcal{A}_f$, i.e., if $\mathcal{H}$ makes the diagram
\begin{equation}\label{eq:o/u homotopy}
\xymatrix{ \land V \ar[rr]^-{\mathcal{T}_g} \ar[d]_{I} && \land V'\otimes \land(t, dt) \ar[d]^{I''\otimes 1} \\
\land V \otimes \land W
\ar@{.>}[rr]^-{\mathcal{H}}\ar[rrd]_{\mathcal{T}_f} && \land V'
\otimes \land W' \otimes C \otimes \land(t, dt)
\ar[d]^{P\otimes 1} \\
&& \land V' \otimes \land W' \otimes \land(t, dt)}
\end{equation}
commute, where $\mathcal{T}_g$ and $\mathcal{T}_f$ are stationary
homotopies at $\mathcal{M}_g$ and $\mathcal{A}_f$, respectively.
 We write
$$[\land V \otimes \land W, \land V' \otimes \land W' \otimes C]_{\uo}$$
for the set of homotopy classes of maps under $\mathcal{M}_g$ and
over $\mathcal{A}_f$. We will define an assignment $F \mapsto
\Gamma$ and show that it leads to a bijection of homotopy sets when
$p \colon E \to B$ is rationalized.

We begin by choosing and fixing models for $g$ and $f$.  First, we
claim a model $\mathcal{A}_f$ for $f$ may be constructed as a map
under $\mathcal{M}_g.$    That is, we have $\mathcal{A}_f \circ I =
I' \circ \mathcal{M}_g$. Writing $A_{PL}(X) = (A(X), \delta_X)$ for
the DG algebra of Sullivan polynomial forms, the relative  model of
the map $A(p) \colon (A(B),\delta_B) \to (A(E), \delta_E)$   is of
the form $(A(B) \otimes \land W, D)$ \cite[Sec.14]{FHT}. Choose a
minimal model $\eta_B \colon (\land V, d) \to (A(B), \delta_B)$ (the
Sullivan minimal model of $B$) and obtain a quasi-isomorphism
$$ \eta_E
\colon (\land V \otimes \land W, D) \to (A(E),\delta_E).$$

Now recall the \emph{surjective trick}: Given a graded algebra $U$
we define $(S(U), \delta)$ to be the contractible DG algebra on a
basis $U \oplus\delta(U)$.  Given a DG algebra map $\eta \colon (B,
d) \to (A, d)$, this manoeuvre results in a diagram
$$\xymatrix{B  \ar[d]_{\eta}
\ar@<0.75ex>[r]^-{\alpha} & B\otimes S(A)
\ar@<0.75ex>[l]^-{\beta}
\ar@{>>}[ld]^{\gamma}\\
A}$$
in which $\gamma$ is a surjection, and both
$\alpha$
and $\beta$ are quasi-isomorphisms.

This trick is used in the standard construction of  the model
$\mathcal{M}_g$ for $g.$ Write     $\eta_{B} \colon (\land V,d)  \to
(A(B), \delta_B)$ and $\eta_{B'} \colon (\land V', d') \to (A(B'),
\delta_{B'})$ for Sullivan minimal models for $B$ and $B'$. Convert
$\eta_{B'}$ into a surjection $\gamma_{B'} \colon \land V'  \otimes
S(A(B')) \to A(B')$ as above. We then lift the composite
$A(g)\circ\eta_B$ through the surjective quasi-isomorphism
$\gamma_{B'}$, using  the standard lifting lemma
\cite[Lem.12.4]{FHT}. We thus obtain $\phi_g \colon \land V \to
\land V' \otimes S(A(B'))$. Now set $\mathcal{M}_g= \beta_{B'} \circ
\phi_g$.  All this is summarized in the following diagram.
\begin{displaymath}
\xymatrix{ & & & \land V' \otimes S(A(B'))
\ar@{>>}@/^1pc/[ddl]^-{\gamma_{B'}}_(0.6){\simeq}
\ar@<1ex>[dl]^(0.6){\beta_{B'}}\ar@<1ex>[dl]^(0.4){\simeq} \\
\land V \ar@{.>}[rr]_{\mathcal{M}_g}
\ar@/^1pc/[rrru]^-{\phi_{g}}\ar[d]_{\eta_{B}}^{\simeq} & &
\land V'  \ar[d]_{\eta_{B'}}^{\simeq}
\ar[ur]^-{\alpha_{B'}}\ar[ur]^(0.3){\simeq}\\
 A(B) \ar[rr]_{A(g)} & & A(B')}
\end{displaymath}
Here the symbol $\simeq$  indicates    a map is a quasi-isomorphism.
By construction, we have $\gamma_{B'}\circ\phi_g = A(g)\circ\eta_B$.
We will use the letters $\phi$, $\alpha$, $\beta$, $\gamma$, and
$\eta$, with suitable subscripts, in a consistent way for diagrams
of the above form.  Notice that $\alpha$ is the obvious inclusion,
and $\beta$ is the obvious projection.

We apply the same construction to obtain a model for
$\mathcal{A}_f$. However, in this case we use a relative version of
the lifting lemma. Converting the vertical quasi-isomorphisms in the
commutative diagram
$$\xymatrix{\land V'  \ar[r]^-{I'} \ar[d]_{\eta_{B'}}^{\simeq}& \land V'
\otimes\land W' \ar[d]^{\eta_{E'}}_{\simeq}\\
A(B') \ar[r]_{A(p')} & A(E')}$$
to surjections results in a commutative diagram
$$\xymatrix{\land V'\otimes S(A(B'))  \ar[rr]^-{I'\otimes S(A(p'))} \ar@{>>}[d]_{\gamma_{B'}}^{\simeq}& & \land V'
\otimes\land W'\otimes S(A(E')) \ar@{>>}[d]^{\gamma_{E'}}_{\simeq}\\
A(B') \ar[rr]_{A(p')} & & A(E'),}$$
in which $S(A(p')) \colon  S(A(B')) \to   S(A(E'))$ is the map
induced by $A(p') \colon A(B') \to A(E')$.  We incorporate this into
the following \emph{relative lifting problem}:
\begin{equation}\label{eq:relative lifting}
\xymatrix{\land V  \ar[rr]^-{I'\otimes S(A(p'))\circ \phi_g} \ar[d]_{I}& & \land V'
\otimes\land W'\otimes S(A(E')) \ar@{>>}[d]^{\gamma_{E'}}_{\simeq}\\
\land V \otimes \land W \ar[rr]_{A(f)\circ \eta_E}
\ar@{.>}[rru]^{\phi_f} & & A(E').}
\end{equation}
The relative lifting lemma as in \cite[Lem.14.4]{FHT} provides the
lift $\phi_f$ that makes both upper and lower triangles commute. But
we now break-off from our development of ideas to give here an
extension of that result, and also its development into
\cite[Lem.14.6]{FHT}, both of which which we need for the sequel and
neither of which we can find in the literature.

\begin{proposition}[Under and Over Lifting Lemma]%
\label{prop: under over lifting}%
Suppose given a diagram of DG algebra maps
$$\xymatrix{A  \ar[rr]^-{f} \ar[dd]_{i}& &  B \ar[dd]^(0.7){\gamma}_(0.7){\simeq} \ar[rr]_{r_B} & & B' \ar@/_1pc/[ll]_{i_B} \ar[dd]^{\gamma'}\\
 & & & & \\
A \otimes \land V \ar[rr]_{\phi} \ar@{.>}[rruu]^{\Psi}
\ar[rrrruu]_(0.7){\psi}
 & & C \ar[rr]^{r_C} & & C' \ar@/^1pc/[ll]^{i_C}}$$
that satisfies
\begin{itemize}
\item commutativity: $\gamma\circ f = \phi\circ i$ and $\gamma'\circ
r_B = r_C\circ \gamma$;
\item $i_B$ and $i_C$ are splittings that are ``natural," in that we
have $r_B\circ i_B = 1_{B'}$, $r_C\circ i_C = 1_{C'}$, and also
$\gamma\circ i_B = i_C \circ \gamma'$;
\item $\psi$ is a lift of $r_C \circ \phi$ through $\gamma'$
relative to $r_B\circ f$, in that $r_B\circ f = \psi\circ i$ and
$\gamma' \circ \psi = r_C \circ \phi$;
\item $\gamma$ is a quasi-isomorphism that is onto
$\mathrm{ker}(r_C)$ (not necessarily surjective).
\end{itemize}
Then there exists a lift $\Psi$ of $\phi$ through $\gamma$ that is
under $f$ and over $\psi$: $\gamma\circ\Psi = \phi$, $\Psi\circ i =
f$, and $r_B\circ \Psi = \psi$.
\end{proposition}

\begin{proof}
We assume that $V$ admits a decomposition $V = \oplus_{i\geq 1}
V(i)$, with respect to which $A \otimes \land V$ satisfies the
nilpotency condition $d\big(V(i)\big) \subseteq A \otimes \land V(<
i)$, and proceed by induction on $i$. Induction starts with $i = 0$,
where we already have the lift $f$. Now suppose that $\Psi$ has been
constructed on $A \otimes \land V(< i)$, and that $v \in V(i)$, for
some $i \geq 1$.  Then $dv \in  A \otimes \land V(< i)$ and
so $\Psi(dv)$ is defined.  We have $d(\Psi(dv)) = \Psi(d^2(v)) = 0$,
so $\Psi(dv) \in \mathcal{Z}(B)$.  Furthermore, we have
$\gamma_*([\Psi(dv)]) = 0$, since $\gamma\circ\Psi(dv) = \phi(dv) =
d\phi(v)$.  Since $\gamma$ is a quasi-isomorphism, $\exists b \in B$
with $\Psi(dv) = db$. Now consider $\phi(v) - \gamma(b) \in C$: we
have $d\big(\phi(v) - \gamma(b)\big) = d\phi(v) - \gamma(db) =
\phi(dv) - \gamma\circ \Psi(dv) = 0$, and so $\phi(v) - \gamma(b)
\in \mathcal{Z}(C)$.  Now ``polarize" this cycle, using the
splitting $i_C$, by writing
\begin{equation}\label{eq: phi - eta}
\phi(v) - \gamma(b) = \big[\big(\phi(v) - \gamma(b)\big) - i_C\circ
r_C\big(\phi(v) - \gamma(b)\big)\big] + i_C\circ r_C\big(\phi(v) -
\gamma(b)\big).
\end{equation}
Observe that we have
$$i_C\circ r_C\big(\phi(v) -
\gamma(b)\big) = i_C\circ \gamma' \circ \psi(v) - \gamma\circ
i_B\circ r_B(b) = \gamma\big(i_B\circ \psi(v) - i_B\circ
r_B(b)\big).$$
Also, writing $\chi = \big(\phi(v) - \gamma(b)\big) - i_C\circ
r_C\big(\phi(v) - \gamma(b)\big)$, we have that $\chi \in
\mathrm{ker}(r_C) \cap \mathcal{Z}(C)$.  Since $\gamma$ is a
quasi-isomorphism, $\exists \beta \in \mathcal{Z}(B)$ with
$\gamma_*([\beta]) = [\chi]$, and so $\gamma(\beta) = \chi + d\xi$
for some $\xi \in C$.  Then we have
\begin{align*}
\gamma\big( \beta - i_B\circ r_B(\beta)\big) &= \gamma(\beta) -
\gamma\circ i_B\circ r_B(\beta) = \gamma(\beta) - i_C\circ
r_C\circ\gamma(\beta)\\
&= \chi + d\xi - i_C\circ r_C(\chi + d\xi) = \chi + d\xi - i_C\circ
r_C d\xi\\
&= \chi + d\big(\xi - i_C\circ r_C(\xi)\big).
\end{align*}
Since $\xi - i_C\circ r_C(\xi) \in \mathrm{ker}(r_C)$, and $\gamma$
is onto $\mathrm{ker}(r_C)$, $\exists b' \in B$ with $\gamma(b') =
\xi - i_C\circ r_C(\xi)$.  Without loss of generality, we may choose
$b' \in \mathrm{ker}(r_B)$, since we have $\gamma(b') = \gamma(b') -
i_C\circ r_C\circ \gamma (b') = \gamma\big(b' - i_B\circ r_B(b')
\big)$. So we have $\gamma\big( \beta - i_B\circ r_B(\beta)\big) =
\chi + d \gamma(b')$, or
$$\chi = \gamma\big( \beta - i_B\circ
r_B(\beta) - d b' \big).$$
Substituting this last identity and the one obtained earlier into
(\ref{eq: phi - eta}) now gives
$$
\phi(v) - \gamma(b) = \gamma\big( \beta - i_B\circ r_B(\beta) - d b'
\big) + \gamma\big(i_B\circ \psi(v) - i_B\circ r_B(b)\big),
$$
so we have
$$\phi(v) = \gamma\big( b - i_B\circ
r_B(b) + \beta - i_B\circ r_B(\beta) - d b' + i_B\circ
\psi(v)\big).$$
Now define
$$\Psi(v) =  \big(b - i_B\circ
r_B(b)\big) + \big(\beta - i_B\circ r_B(\beta)\big) - d b' +
i_B\circ \psi(v).$$
Observe that $\Psi(v) - i_B\circ \psi(v) \in \mathrm{ker}(r_B)$.
Evidently, we have $\gamma\circ\Psi(v) = \phi(v)$, and
$r_B\circ\Psi(v) = r_B\circ i_B\circ\psi(v) = \psi(v)$ as desired.
Induction is complete, and the result follows.
\end{proof}

We extend this to a result on lifting homotopy classes in the
following (cf.~\cite[Prop.12.9, Prop.14.6]{FHT},
\cite[Prop.II.2.11]{Baues}, and \cite[Prop.A.4]{LS}).

\begin{proposition}[Under and Over Homotopy Lifting Lemma]%
\label{prop: under over homotopy lifting}%
Suppose given a commutative diagram of DG algebra maps
$$\xymatrix{A  \ar[r]^-{g} \ar[d]_{i}&  B \ar[d]_{r_B} \ar@{->>}[r]^{\gamma}_{\simeq}  & C \ar[d]_{r_C}\\
A \otimes \land V \ar[r]_-{f} \ar@{.>}[ru]
 & B' \ar@/_1pc/[u]_{i_B} \ar[r]^{\gamma'} & C' \ar@/_1pc/[u]_{i_C}}$$
with $\gamma \colon B \to C$ a surjective quasi-isomorphism and
$i_B$, $i_C$ natural splittings as in \propref{prop: under over
lifting}. That is, we have $r_B\circ i_B = 1_{B'}$, $r_C\circ i_C =
1_{C'}$, and also $\gamma\circ i_B = i_C \circ \gamma'$. Then
$\gamma$ induces a bijection of homotopy sets
$$\lambda_*\colon [A \otimes \land V, B]_{\uo} \to [A \otimes \land V, C]_{\uo}.$$
\end{proposition}

\begin{proof}
Homotopy classes in $[A \otimes \land V, B]_{\uo}$ are represented
by maps $F$ that make the left-hand diagram below commute, and
homotopies $H$ between two such maps make the right-hand diagram
below commute:
$$\xymatrix{A \ar[rd]^{g} \ar[d]_{i} \\
A \otimes \land V \ar[r]_-{F} \ar[rd]_{f} & B \ar[d]^{r_B}\\
 & B'} \qquad
\xymatrix{A \ar[rd]^{T_g} \ar[d]_{i} \\
A \otimes \land V \ar[r]_-{H} \ar[rd]_{T_f} & B\otimes \land(t, dt) \ar[d]^{r_B\otimes1}\\
 & B'\otimes \land(t, dt)}$$
Homotopy classes in  $[A \otimes \land V, C]_{\ou}$ are represented
by maps $\mathcal{F}$ that make the left-hand diagram below commute,
and homotopies $\mathcal{H}$ between two such maps make the
right-hand diagram below commute:
$$\xymatrix{A \ar[rd]^{\gamma\circ g} \ar[d]_{i} \\
A \otimes \land V \ar[r]_-{\mathcal{F}} \ar[rd]_{\gamma'\circ f} & C \ar[d]^{r_C}\\
 & C'} \qquad
\xymatrix{A \ar[rd]^{T_{\gamma\circ g}} \ar[d]_{i} \\
A \otimes \land V \ar[r]_-{\mathcal{H}} \ar[rd]_{T_{\gamma'\circ f}} & C\otimes \land(t, dt) \ar[d]^{r_C\otimes1}\\
 & C'\otimes \land(t, dt)}$$
In these diagrams, $T_g$ denotes the stationary homotopy at $g$, for
example. Evidently, we have $(\gamma'\otimes 1)\circ T_f =
T_{\gamma'\circ f}$ and $(\gamma\otimes 1)\circ T_g = T_{\gamma\circ
g}$.  So $\gamma_*$ gives a well-defined map of under-and-over
homotopy classes.  Now suppose given some $[\phi] \in [A \otimes
\land V, C]_{\uo}$.  We have the following diagram,
$$\xymatrix{A  \ar[r]^-{g} \ar[d]_{i}&  B \ar@{->>}[d]^(0.7){\gamma}_(0.7){\simeq} \ar[r]^{r_B} &  B' \ar@/_1pc/[l]_{i_B} \ar[d]^{\gamma'}\\
A \otimes \land V \ar[r]_{\phi} \ar@{.>}[ru] \ar[rru]_(0.7){f}
 & C \ar[r]^{r_C} & C' \ar@/^1pc/[l]^{i_C},}$$
and it follows from \propref{prop: under over lifting} that
$\gamma_*$ is surjective.

Finally, we show that $\gamma_*$ is injective.  Suppose we have two
maps $F_0, F_1 \colon A \otimes \land V \to B$ under $g$ and over
$f$, and $\gamma\circ F_0$ and $\gamma\circ F_1$ are homotopic via a
homotopy $\mathcal{H}\colon A \otimes \land V \to C\otimes \land(t,
dt)$ under $\gamma\circ f$ and over $\gamma'\circ g$.  Form the
commutative cube
\begin{displaymath}
\xymatrix{ & Q
  \ar[rr]
\ar'[d][dd] & & B' \times B'
\ar[dd]^-{\gamma'\times\gamma'} \\
P \ar[ru]^{r}\ar[rr]
  \ar[dd]& &
B\times B \ar[ru]^(0.4){r_B\times r_B}
\ar[dd]^(0.3){\gamma\times \gamma}\\
  &  C'\otimes \land(t, dt) \ar'[r]_-{(1\cdot\varepsilon_0, 1\cdot\varepsilon_1)}[rr] &
&
C'\times C'  \\
C\otimes \land(t, dt) \ar[ru]^{r_C\otimes1}
\ar[rr]_{(1\cdot\varepsilon_0, 1\cdot\varepsilon_1)} & & C\times
C\ar[ru]_{r_C\times r_C} }
\end{displaymath}
in which the front and back faces are pullbacks, so that
$$P = \big(C\otimes\land(t, dt)\big) \times_{C\times C} \big(B \times B\big) \qquad \mathrm{and} \qquad
Q = \big(C'\otimes\land(t, dt)\big) \times_{C'\times C'} \big(B'
\times B'\big).$$
The map $r\colon P \to Q$ is that induced on the pullbacks so as to
make the cube commute.  Since the forwards maps are composed of
$r_C$ and $r_B$, they admit natural splittings and a natural
splitting $i\colon Q \to P$ of $r$ is induced. Denote by $\Gamma$
the map induced from the pullback diagram as follows:
$$\xymatrix{
B \otimes \land(t, dt) \ar@/^1pc/@{->>}[rrd]^{(1\cdot\varepsilon_0,
1\cdot\varepsilon_1)} \ar[rd]_{\Gamma}
\ar@/_1pc/@{->>}[ddr]_{\gamma\otimes1}^{\simeq}\\
 & P \ar[r] \ar[d]^{\simeq} & B \times B
 \ar@{->>}[d]^{\gamma\times\gamma}_{\simeq}\\
 & C\otimes\land(t, dt) \ar@{->>}[r]_{(1\cdot\varepsilon_0,
1\cdot\varepsilon_1)} & C \times C }$$
It follows from properties of the pullback that $\Gamma$ is a
surjective quasi-isomorphism.  We now include $\Gamma$, and the
corresponding map $\Gamma' \colon B' \otimes \land(t, dt) \to Q$
obtained from the back face of the above pullback cube, into the
right-hand part of the following diagram:
$$\xymatrix{A  \ar[rr]^-{T_g} \ar[dd]_{i}& &  B\otimes\land(t, dt)
 \ar@{->>}[dd]^(0.7){\Gamma}_(0.7){\simeq} \ar[rr]_{r_B\otimes1} & & B'\otimes\land(t, dt) \ar@/_1pc/[ll]_{i_B\otimes1} \ar[dd]^{\Gamma'}\\
 & & & & \\
A \otimes \land V \ar[rr]_{\big(\mathcal{H}, (F_0, F_1)\big)}
\ar@{.>}[rruu] \ar[rrrruu]_(0.7){T_f}
 & & C \ar[rr]^{r} & & C' \ar@/^1pc/[ll]^{i}}$$
The lift obtained from \propref{prop: under over lifting} gives the
desired homotopy from $F_0$ to $F_1$ under $g$ and over $f$.
\end{proof}

Now we return to the development of ideas that preceded
\propref{prop: under over lifting}. In diagram (\ref{eq:relative
lifting}), \propref{prop: under over lifting} applied to the case in
which $B' = C' = \mathbb{Q}$, and $r_B$, $r_C$, and $\psi$ are the
augmentations, provides the relative lift. That is, we obtain a
lifting $\phi_f$ in diagram (\ref{eq:relative lifting}) and set
$\mathcal{A}_f = \beta_{E'} \circ \phi_f.$  It follows from the
definitions that we have $I'\circ \mathcal{M}_g = \mathcal{A}_f\circ
I$.

Having chosen and fixed models for $f$ and $g$, we now extend their
construction to Sullivan models for maps $F$ over $g$ and under $f$
as in (\ref{eq:over/under}).

\begin{proposition}\label{prop:over/under}
A Sullivan model $\mathcal{A}_F$ for a map $F$ that makes
(\ref{eq:over/under}) commute may be chosen so that $\Gamma =
\mathcal{A}_F$ makes (\ref{eq:o/u}) commute. Further, if $F_0$ and
$F_1$ are homotopic over $g$ and under $f$, then $\mathcal{A}_{F_0}$
and $\mathcal{A}_{F_1}$ are DG homotopic under $\mathcal{M}_g$ and
over $\mathcal{A}_f.$
\end{proposition}

\begin{proof}
The existence of a model $\mathcal{A}_F$ of the desired form follows
directly from \propref{prop: under over lifting}. Apply the Sullivan
functor $A(-)$ to diagram (\ref{eq:over/under}), and incorporate the
result, along with the models just constructed, into the following
diagram:
$$\xymatrix{\land V  \ar[rr]^-{\big(I''\otimes S(A(q))\big)\circ\phi_g} \ar[dd]_{I}& &  \land V' \otimes \land W' \otimes C \otimes S(A(E' \times Z))
\ar[dd]^(0.7){\gamma_{E' \times Z}}_(0.7){\simeq} \ar[rr]_-{P\otimes
S(A(i))} & & \land V' \otimes \land W' \otimes S(A(E'))
 \ar@/_1pc/[ll]_{J\otimes
S(A(p_1))} \ar[dd]^{\gamma_{E'}}\\
 & & & & \\
\land V \otimes \land W  \ar[rr]_{A(F)\circ \eta_E}
\ar@{.>}[rruu]^{\phi_F} \ar[rrrruu]_(0.7){\phi_f}
 & & A(E'\times Z) \ar[rr]^{A(i)} & & A(E') \ar@/^1pc/[ll]^{A(p_1)}}$$
\propref{prop: under over lifting} gives an under-over lift
$\phi_F$, and we set $\mathcal{A}_F = \beta_{E'\times Z}\circ
\phi_F$.

We show the relation of homotopy over and under is preserved through
passing to models in two steps.  First, we establish that there is a
well-defined map of homotopy classes
$$[E'\times Z, E]_{\ou} \to [\land V \otimes \land W, A(E'\times
Z)]_{\uo}.$$
Suppose $H \colon E' \times Z \times I \to E$ is a homotopy from
$F_0$ to $F_1$ which is over $g$ and under $f$, as in diagram
(\ref{eq:over/under homotopy}).  Apply $A(-)$ to that diagram, and
adapt the argument of \cite[Pro.12.6]{FHT} as follows (some of the
notation in what follows is adopted from \cite[Pro.12.6]{FHT}). From
the diagram
$$\xymatrix{A(E' \times Z) \otimes
\land(t, dt) \ar[dd]_(0.6){A(p^{E'\times Z})\cdot
A(p^I)}_(0.4){\simeq} \ar[rrr]_{A(i)\otimes1}
\ar[rd]^{(\mathrm{id}\cdot\varepsilon_0,
\mathrm{id}\cdot\varepsilon_1)} & & & A(E') \otimes \land(t, dt)
\ar@/_1pc/[lll]_{A(p_1)\otimes1} \ar[dd]^{A(p^{E'})\cdot A(p^I)}
\ar[ld]^{(\mathrm{id}\cdot\varepsilon_0, \mathrm{id}\cdot\varepsilon_1)}\\
 & A(E' \times Z) \times A(E' \times Z) &A(E') \times A(E') \ar[l] & \\
A(E'\times Z\times I) \ar[rrr]^-{A(i\times1)} \ar[ru]^{(A(j_0),
A(j_1))} & & & A(E'\times I) \ar@/^1pc/[lll]^{A(p_1\times 1)}
\ar[lu]_{(A(j_0), A(j_1))}}$$
we adjust the left-hand vertical map into a map $\gamma$ that is
onto the kernel of $A(i\times 1)$, using $U = \mathrm{ker}(A(j_0),
A(j_1)) \cap \mathrm{ker}(A(i\times 1))$. Now apply \propref{prop:
under over lifting} to the diagram
$$\xymatrix{\land V  \ar[rr]^-{(A(g\circ q)\otimes1)\circ \eta_B} \ar[dd]_{I}& &  A(E' \times Z) \otimes
\land(t, dt)\otimes S(U) \ar[dd]_(0.6){\gamma}_(0.4){\simeq}
\ar[rr]_{A(i)\otimes1} & & A(E') \otimes \land(t, dt)
 \ar@/_1pc/[ll]_{A(p_1)\otimes1} \ar[dd]^{A(p^{E'})\cdot A(p^I)}\\
 & &  & & \\
\land V \otimes \land W  \ar[rr]_{A(H)\circ \eta_E}
\ar@{.>}[rruu]^{\phi_H} \ar[rrrruu]_(0.7){A(f)\circ \eta_E}
 & & A(E'\times Z\times I) \ar[rr]^{A(i\times1)} & & A(E'\times I) \ar@/^1pc/[ll]^{A(p_1\times 1)}}$$
to obtain a DG homotopy $\mathcal{G} = \beta \circ \phi_H$ from
$\eta_E\circ A(F_0)$ to $\eta_E\circ A(F_1)$ that is a DG homotopy
under $A(g)\circ \eta_B$ and over $A(f) \circ \eta_E$. Thus far, we
have established that there is a well-defined map of homotopy
classes
$$\mathcal{S}\colon [E'\times Z, E]_{\ou} \to [\land V \otimes \land W, A(E'\times
Z)]_{\uo}.$$

We now want to lift this correspondence up to minimal models.
Converting the quasi-isomorphism $\eta_{E'\times Z} \colon \land
V'\otimes\land W' \otimes C \to A(E'\times Z)$ to a surjection
gives, amongst other data, a surjective quasi-isomorphism
$\gamma_{E'\times Z} \colon \land V'\otimes\land W' \otimes C
\otimes S\big(A(E'\times Z)\big) \to A(E'\times Z)$, and the
retraction map $\beta \colon \land V'\otimes\land W' \otimes C
\otimes S\big(A(E'\times Z)\big) \to \land V'\otimes\land W' \otimes
C$ which we observe is also a surjective quasi-isomorphism.

The diagrams
$$\xymatrix{\land V  \ar[r]^-{g} \ar[d]_{i}&  \land V'\otimes\land W' \otimes C \otimes S\big(A(E'\times Z)\big) \ar[d]_{P\otimes S(A(i))}
 \ar@{->>}[r]^-{\gamma_{E'\times Z}}_-{\simeq}  & A(E'\times Z) \ar[d]_{A(i)}\\
\land V \otimes \land W \ar[r]_-{f} \ar@{.>}[ru]
 & \land V'\otimes\land W' \otimes S\big(A(E')\big) \ar@/_1pc/[u]_{J\otimes S(A(p_1))} \ar[r]^-{\gamma_{E'}} & A(E') \ar@/_1pc/[u]_{A(p_1)}}$$
and
$$\xymatrix{\land V  \ar[r]^-{g} \ar[d]_{i}&  \land V'\otimes\land W' \otimes C \otimes S\big(A(E'\times Z)\big) \ar[d]_{P\otimes S(A(i))}
 \ar@{->>}[r]^-{\beta_{E'\times Z}}_-{\simeq}  & \land V'\otimes\land W' \otimes C \ar[d]_{P}\\
\land V \otimes \land W \ar[r]_-{f} \ar@{.>}[ru]
 & \land V'\otimes\land W' \otimes S\big(A(E')\big) \ar@/_1pc/[u]_{J\otimes S(A(p_1))} \ar[r]^-{\beta_{E'}} & \land V'\otimes\land W' \ar@/_1pc/[u]_{J}}$$
yield, respectively, bijections
$$(\gamma_{E'\times Z})_* \colon [\land V \otimes \land W, \land V'\otimes\land W' \otimes C \otimes S\big(A(E'\times
Z)\big)]_{\uo} \to [\land V \otimes \land W, A(E'\times Z)]_{\uo}$$
and
$$(\beta_{E'\times Z})_* \colon [\land V \otimes \land W, \land V'\otimes\land W' \otimes C \otimes S\big(A(E'\times
Z)\big)]_{\uo} \to [\land V \otimes \land W, \land V'\otimes\land W'
\otimes C]_{\uo}.$$
Combined with our previous work, we now have a well-defined map of
homotopy classes
$$(\beta_{E'\times Z})_*\circ ((\gamma_{E'\times Z})_*)^{-1}\circ \mathcal{S}\colon
[E'\times Z, E]_{\ou} \to [\land V \otimes \land W, \land
V'\otimes\land W' \otimes C]_{\uo}$$
induced by the assignment $F \mapsto \mathcal{A}_F$.
\end{proof}

Now we turn to rationalization of the above correspondence. Consider
the diagram
\begin{equation} \label{eq:rat}  \xymatrix{ E' \ar[r]^{f_\Q} \ar[d]_{p'} & E_\Q \ar[d]^{p_\Q} \\
B' \ar[r]^{g_\Q} & B_\Q }\end{equation} obtained from
(\ref{eq:fibsquare}) by replacing $p, f$ and $g$ by their
rationalizations. (Recall we are assuming the spaces $E, E$ and $B,
B'$ are all simply connected.) Let $$\Psi \colon [E' \times Z,
E_\Q]_{\ou} \to [\land V \otimes \land W, \land V' \otimes \land W'
\otimes C]_{\uo}$$ denote the corresponding map. The following
result extends the standard correspondence between homotopy classes
of maps into a rational space and DG algebra homotopy classes of
maps between Sullivan models \cite[Pro.12.7,  Pro.17.13]{FHT}.

\begin{proposition} \label{prop:o/u}
Suppose the spaces $ E'$ and $Z$ are finite CW complexes.
Then $\Psi$ is a  bijection of sets.
\end{proposition}
\begin{proof}
First assume $p_\Q \colon E_\Q \to B_\Q$ is a principal $K(W_n,
n)$-fibration for $W_n$  a  rational space concentrated in degree
$n$. Let $F \colon E' \times Z \to E_\Q$ be a map over $g_\Q$ and
under $f_\Q$. Define a class $\gamma_F \in H^n(Z; W_m) \cong
\Hom(W_m, H^*(Z; \Q))$ as follows. Let   $j \colon  Z  \to E' \times
Z$ denote the inclusion and observe $p_\Q \circ F \circ j \simeq *.$
Thus $F \circ j$  is homotopic to a map $G_F  \colon Z \to K(W_n,
n)$, or equivalently, a class $\gamma_F \in H^n(Z; W_m)$.

Conversely, given a class $\gamma \in H^n(Z; W_m)$ we construct a map $F_\gamma$
 over $g$ and under $f$ as follows.
Let $\mathcal{P} \colon E_\Q \times K(W_n, n) \to E_\Q$ denote the
fibrewise action. Here $\mathcal{P}$ is  a map over $1_{B_\Q}$ and
under the inclusion $ E_\Q \to E_\Q \times K(W_n, n)$. Let $G \colon
Z \to K(W_n, n)$ be the map induced by $\gamma$.
 Define
$F_\gamma$ to be the composite
 $$ \xymatrix{   E' \times Z   \ar[rr]^{   f_\Q  \times G} && E_\Q
\times K(W_n, n) \ar[r]^{\ \ \ \ \ \ \  \mathcal{P}} & E_\Q}.
$$
It is direct to check that  $F_\gamma$  is a map over $g_\Q$ and
under $f_\Q$ and that   the assignments  $F \mapsto \gamma_F$ and
$\gamma \mapsto F_\gamma$ set up a bijection $$ [E' \times Z,
E_\Q]_{\ou} \equiv H^n(Z; W_n).$$

Now suppose $\Gamma$ is a DG map  making (\ref{eq:o/u}) commute. Then we may  write
$ \Gamma(\chi) =\mathcal{A}_f(\chi) + \eta(\chi)$  where $\eta(\chi) = 0$ for $\chi \in \land V$ and $P(\eta(\chi)) = 0.$ Given
 $w \in W_n,$ since $D(w) \in \land V$ it follows that $P'(\eta(w))$ is a cycle of $ C.$
   Let $\gamma_\Gamma \in H^n(Z; W_n)$ denote the class corresponding to $P' \circ \eta$
restricted to $W_n$ where  $P' \colon \land V' \otimes \land W' \otimes C \to C$ is the projection.  The assignment $\Gamma \mapsto \gamma_\Gamma$
then gives a well-defined surjection
$$\xymatrix{ [\land V \otimes \land W, \land V' \otimes \land W' \otimes C]_{\uo} \ar@{>>}[r] &  H^n(Z; W_n)}.$$
Finally,  observe that the spatial realization $F$ of
$\gamma_\Gamma$ constructed in the preceding paragraph has Sullivan
model $\Gamma.$ This is direct from the fact that $\mathcal{P}$ is a
map over $1_{B_\Q}$ and under   $E_\Q \to E_\Q \times K(W_n, n).$
Suppose $\Gamma_0$ and $\Gamma_1$ satisfy $\gamma_{\Gamma_0} =
\gamma_{\Gamma_1} \in H^n(Z; W_n).$ Writing $\gamma$ for this class,
 we obtain a map $F_\gamma$ over $g_\Q$ and under $f_\Q$ with two Sullivan models $\Gamma_0$ and $\Gamma_1.$ By Proposition \ref{prop:over/under},
 $\Gamma_0$ and $\Gamma_1$ are homotopic under $\mathcal{M}_g$ and over $\mathcal{A}_f$ and the result is proved in this case.

Now proceed by induction over a Moore-Postnikov factorization of the
fibration $ p_\Q \colon E_\Q \to B_\Q.$ Let $(p_\Q)_n \colon
(E_\Q)_n \to (E_\Q)_{n-1}$ be the $n$th fibration, a principal
fibration with fibre $K(W_n, n)$. The Sullivan model for $(p_\Q)_n$
is of the form $\land V \to \land V \otimes \land W_{(n)}$ where
$W_{(n)} = \bigoplus_{k \leq n}W_k.$ Since $E' \times Z$ is finite,
for $n$ large, composition with the canonical map $h_n \colon E_\Q
\to (E_n)_\Q$ yields a bijection $$[E' \times Z, E_\Q]_\ou \equiv
[E' \times Z, (E_n)_\Q]_\ou.$$ Similarly, the inclusion $W_{(n)} \to
W$ for such $n$ induces a bijection $$ [\land V \otimes \land W,
\land V' \otimes \land W' \otimes C]_\uo\equiv [\land V \otimes
\land W_{(n)}, \land V' \otimes \land W' \otimes C]_\uo.$$
  \end{proof}

We apply the foregoing to define $\Phi(\alpha)$ for    an element
$\alpha \in \pi_n(\map_g(E', E;f))$.   Let $F \colon E' \times S^n
\to E$ be the adjoint of $\alpha.$ By  Proposition
\ref{prop:over/under}, $F$ has a Sullivan model of the form
$$\mathcal{A}_F \colon \land V \otimes \land W \to \land V'
\otimes \land W' \otimes ( \land(u) /\langle u^2 \rangle) \hbox{\,
with \, } \mathcal{A}_F(\chi) = \mathcal{A}_f(\chi)  + u
\theta(\chi)$$    for $\chi \in \land V \otimes \land W$, and where
$\mathcal{A}_f(\chi) = \mathcal{M}_g(\chi)$ and $\theta(\chi) = 0$
for $\chi \in \land V$.   Here  $( \land(u) /\langle u^2 \rangle,
0)$ is a Sullivan model for $S^n$ with $|u| = n.$ The map $\theta$
is thus  linear of degree $n$ vanishing on $\land V$.  The following
facts are standard and direct to check:
\begin{itemize}
\item[(1)] $ \mathcal{A}_F(\chi_1  \chi_2) = \mathcal{A}_F(\chi_1)  \mathcal{A}_F(\chi_2) \Longrightarrow \ \    \theta $ is an
$\mathcal{A}_f\hbox{-derivation.}$
\\

\item[(2)] $\mathcal{A}_F \circ D = D \circ \mathcal{A}_F$  $ \Longrightarrow  \ \     \theta$ is an
$\mathcal{A}_f\hbox{-derivation cycle.}$
\end{itemize}

\begin{lemma}
The homology class $\langle \theta \rangle \in H_n(\der_{\land
V}(\land V \otimes \land W, \land V' \otimes \land W';
\mathcal{A}_f))$ is independent of the choice of representative of
$\alpha \in \pi_n(\map_g(E', E;f))$.
\end{lemma}

\begin{proof}
Suppose that $a, b \colon S^n \to \map_g(E', E; f)$ both represent
$\alpha \in \pi_n(\map_g(E', E; f))$.  Observe that a homotopy from
$a$ to $b$ has adjoint $H \colon E' \times S^n \times I \to E$
giving rise to a commutative diagram as in (\ref{eq:over/under
homotopy}), with $Z = S^n$. This translates into a Sullivan model
$\mathcal{H}$ for $H$ that fits into the following commutative
diagram.
$$
\xymatrix{ \land V \ar[rr]^-{\mathcal{T}_g} \ar[d]_{I} && \land V'\otimes \land(t, dt) \ar[d]^{I''\otimes 1} \\
\land V \otimes \land W
\ar@{.>}[rr]^-{\mathcal{H}}\ar[rrd]_{\mathcal{T}_f} && \land V'
\otimes \land W' \otimes ( \land(u) /\langle u^2 \rangle) \otimes
\land(t, dt)
\ar[d]^{P\otimes 1} \\
&& \land V' \otimes \land W' \otimes \land(t, dt)}
$$
For $\chi \in \land V \otimes \land W$, we may write (being
particular with the order of terms in each sum)
$$\mathcal{H}(\chi) = \mathcal{A}_f(\chi) + \sum_{i \geq 0}\, u\, t^i\, \psi_i(\chi)
+  \sum_{i \geq 0}\, u\, t^i\, dt\, \phi_i(\chi),$$
for linear maps $\psi_i$ and $\phi_i$.  Because the homotopy
$\mathcal{H}$ is under $\mathcal{M}_g$, we have $\psi_i(V) = 0$ and
$\phi_i(V) = 0$ for each $i$.   From our definition above, and since
the original homotopy was from $a$ to $b$, it follows that $\psi_0$,
respectively $\sum_{i \geq 0} \psi_i$, is the derivation $\theta_a$,
respectively $\theta_b$, that corresponds under $\Phi'$ to $[a]$,
respectively $[b]$.  We must show that $\theta_b - \theta_a =
\sum_{i \geq 1} \psi_i$ is a boundary.  In fact, the identity
$\mathcal{H}(\chi_1 \chi_2) = \mathcal{H}(\chi_1)
\mathcal{H}(\chi_2)$ easily yields that each $\psi_i$ and each
$\phi_i$ is an $\mathcal{A}_f$-derivation.  Then, the identity
$\mathcal{H} \circ D = D \circ \mathcal{H}$ yields that
$$\mathcal{D}(\phi_i) = [d, \phi_i] = (i+1) \psi_{i+1},$$
for each $i \geq 0$.  Hence we have
$$\theta_b - \theta_a =
\sum_{i \geq 0} \frac{1}{i+1} \mathcal{D}(\phi_i),$$
and the cohomology class is well-defined.
\end{proof}

Define $$\Phi' \colon \pi_n(\map_g(E', E; f)) \to H_n(\der_{\land
V}(\land V \otimes \land W, \land V' \otimes \land W';
\mathcal{A}_f))$$ by $ \Phi'(\alpha) =  \langle \theta \rangle.$ The
following result contains the first assertion of Theorem
\ref{thm:main2}.

\begin{theorem}\label{thm:main2'} Suppose given a commutative diagram
$$ \xymatrix{E' \ar[d]_{p'} \ar[r]^f & E \ar[d]^{p} \\ B' \ar[r]^{g} & B} $$
with vertical maps  fibrations and all spaces simply connected CW
complexes. The map $$\Phi' \colon \pi_n(\map_g(E', E; f)) \to
H_n(\der_{\land V}(\land V \otimes \land W, \land V' \otimes \land
W'; \mathcal{A}_f))$$ defined above is a homomorphism for $n \geq
2$.  If $E'$ is finite the rationalization $$ \Phi \colon
\pi_n(\map_g(E', E; f)) \otimes \Q  \longrightarrow H_n(\der_{\land
V}(\land V \otimes \land W, \land V' \otimes \land W';
\mathcal{A}_f))$$ of $\Phi'$ is an isomorphism for $n \geq 2$.
\end{theorem}

\begin{proof}
To show that $\Phi'$ is a  homomorphism for $n \geq 2,$  let $\alpha, \beta \in \pi_n(\map_g(E', E; f))$  with adjoints $F,G \colon E' \times S^n \to E.$ Let
$(F  \mid  G) \colon E' \times (S^n \vee S^n) \to E$ be the map induced by $F$ and $G$. We then have a commutative diagram   (\ref{eq:over/under}) with $Z = S^n \vee S^n$.
A Sullivan model for  $S^n \vee S^n$ is $(\land(u, v)/\langle u^2, uv, v^2 \rangle, 0)$  with $|u|=|v| =n$. Applying  Proposition \ref{prop:over/under}, we see $(F\mid G)$ has Sullivan model
$$\chi \mapsto \chi + u\theta_a(\chi) + v\theta_b(\chi) \: \land V \otimes \land W \to \land V' \otimes \land W' \otimes  (\land(u, v)/\langle u^2, uv, v^2 \rangle)$$
for $\chi \in \land V \otimes \land W.$  Using that  $(F  \mid  G)
\circ (1_{E'} \times i_1) = F$  and $(F  \mid  G) \circ (1_{E'}
\times i_2) = G,$ where $i_j \colon S^n \to S^n \vee S^n$ are the
inclusions, gives $\theta_a $ and  $\theta_b$ are cycle
representatives for  $\Phi'(\alpha)$ and $\Phi'(\beta)$,
respectively. The map  $(F  \mid  G)\circ (1_{E'} \times \sigma)$ is
adjoint to the sum $\alpha + \beta$  where $ \sigma \colon    S^n
\to   S^n \vee S^n$ is the pinch map.    The result now follows from
the fact that a Sullivan model  for $  \sigma$
   is given by
   $$  u, v \mapsto  w  \colon
    \land(u, v)/\langle u^2, uv, v^2 \rangle
   \to \land(w)/\langle w^2 \rangle.$$

Now assume $E'$ is finite. By  Proposition \ref{pro:CW},
composition with a rationalization $\ell_E \colon E \to E_\Q$ gives
a rationalization   $ \ell_E \colon  \map_g(E',E; f) \to
\map_{g_\Q}( E', E_\Q; f_\Q)$.  We take $$\Phi  \colon
\pi_n(\map_{g_\Q}( E', E_\Q; f_\Q)) \to H_n(\der_{\land V}(\land V
\otimes \land W, \land V ' \otimes \land W'; \mathcal{A}_f))$$ to be
the map $\Phi'$ corresponding to the diagram (\ref{eq:rat}).

Suppose $ \Phi(\alpha) = 0$ for $\alpha \in \pi_n(\map_{g_\Q}( E', E_\Q; f_\Q))$.  Then
by Proposition \ref{prop:over/under}, a Sullivan model $\mathcal{A}_F$ for the adjoint $F \colon E' \times S^n \to E_\Q$
is given by:
$$ \chi \mapsto \mathcal{A}_f(\chi) + u\theta(\chi) \colon \land V \otimes \land W \to
\land V' \otimes \land W' \otimes (\land( u)/\langle u^2 \rangle)$$
and $\theta = \D(\overline{\theta})$ for some $\overline{\theta} \in \der^{n+1}_{\land V}(\land V \otimes \land W, \land V' \otimes \land W'; \mathcal{A}_f).$
Define a DG algebra homotopy $\mathcal{H}$  from the map $$\chi \mapsto \mathcal{A}
_f(\chi) \colon \land V \otimes \land W \to \land V' \otimes \land W' \otimes (\land( u)/\langle u^2 \rangle)$$ to the map $\mathcal{A}_F$ by the rule
$$ \chi \mapsto \mathcal{A}_f(\chi) +  tu\theta(\chi) +  (-1)^n dt u \overline{\theta}(\chi) \colon
 \land V \otimes \land W \to
\land V' \otimes \land W' \otimes (\land( u)/\langle u^2 \rangle)
\otimes \land(t, dt).$$ Proposition \ref{prop:o/u} with $Z = S^n$
gives a homotopy $H \colon E' \times S^n \times I \to E_\Q$, over
$g_\Q$ and under $f_\Q$, between the adjoint of the trivial class in
$\pi_n(\map_{g_\Q}(E', E_\Q; f_\Q))$ and $F.$  It follows that
$\Phi$ is injective.

Finally, to prove $\Phi$ is onto for $n \geq 2,$ let $\theta \in
\der^{n}_{\land V}(\land V \otimes \land W, \land V' \otimes \land
W'; \mathcal{A}_f)$ be a cycle. Define a DG algebra map $\Gamma$ by
$$ \chi \mapsto \mathcal{A}_f(\chi) + u\theta(\chi) \colon \land V \otimes \land W \to \land V' \otimes \land W' \otimes C.$$
Using Proposition \ref{prop:o/u} with $Z = S^n,$ we obtain a map $F
\colon E' \times S^n \to E_\Q$ that is under $f_\Q$ and over $g_\Q$.
The  adjoint of $F$  is a class $\alpha \in \pi_n(\map_{g_\Q}(E',
E_\Q; f_\Q))$ with $\Phi(\alpha) = \langle \theta \rangle.$
\end{proof}

We note that $\pi_1(\map_g(E', E; f))$ is not, in general, abelian
and so $\Phi$ cannot, in general,  be   an isomorphism (cf.
\cite[Ex.1.1]{LS2}).  The proof  that  $\Phi'$ is a homomorphism
breaks down if $n=1$ because   $Z = S^1 \vee S^1$ is a non-nilpotent
space. In fact, $\Phi$ is generally not a homomorphism in degree
$1$.

By Proposition \ref{pro:CW}, $\pi_1(\map_g(E', E; f)$ is a nilpotent group though and thus has
a well-defined rank.  We have:

\begin{theorem} With hypotheses as in Theorem \ref{thm:main2'}, if   $E'$  is finite then $$ \mathrm{rank} (\pi_1(\map_g(E', E; f)))   = \mathrm{dim}( H_1(\der_{\land V}(\land V \otimes \land W, \land V' \otimes \land W'; \mathcal{A}_f)).$$
  \end{theorem}
  \begin{proof}  The proof is   an  adaptation of     \cite[Th.1]{LS2} similar to the preceding result.
  Here we use  a Moore-Postnikov factorization of $p_\Q \colon E_\Q \to B_\Q$
  in place of  the  absolute Postnikov tower of $Y$ for $\map(X, Y; f)$ used there.
  Also we use Theorem \ref{thm:main2'} in place of \cite[Th.2.1]{LS}. We omit the details.
   \end{proof}
\section{The rational Samelson Lie Algebra
 of $\aut(p)_\circ$}
\label{sec:main}
In this section, we sharpen Theorem \ref{thm:main2'} in the case $f$ and $g$ are the respective identity maps to prove Theorem \ref{thm:main}.
Fix a fibration $p\colon  E  \to B$ of simply connected CW complexes with $E$ finite.   Observe     $\aut(p)_\circ  = \map_{1_B}(E, E; 1_E)$.
We prove the map $\Phi' \colon \pi_1(\aut(p)_\circ) \to H_1(\der_{\land V} (\land V \otimes \land W))$   induces an isomorphism  after rationalization. We then     show $\Phi' $  induces an isomorphism $$\Phi \colon \pi_*(\aut(p)_\circ) \to H_*(\der_{\land V} (\land V \otimes \land W))$$ of graded Lie algebras.

Let $\alpha \in \pi_p(\aut(p)_\circ)$ and $\beta \in \pi_q( \aut(p)_\circ)$ be homotopy classes with adjoints $F \colon E \times S^p \to E$ and $G \colon E \times S^q \to E$.
Let $\theta_a \in \der^p_{\land V}(\land V \otimes \land W)$ and $\theta_b \in \der^q_{\land V}(\land V \otimes \land W)$ be cycle representatives for $\Phi'(\alpha)$ and $\Phi'(\beta),$ respectively. Define a homotopy class $\alpha * \beta \in [S^p \times S^q, \aut(p)_\circ]$ to be the composite
$$\xymatrix{  S^p \times S^q  \ar[r]^{\! \! \! \! \! \! \! \! \! \! \! \! \! \! \! \alpha\times \beta} & \aut(p)_\circ \times \aut(p)_\circ \ar[r]^{ \ \ \ \ \ \ \mu} & \aut(p)_\circ}$$
where $\mu$ is   the  multiplication (composition of maps) in $\aut(p)_\circ$.
Write $$F *G \colon E \times S^p \times S^q \to E$$
for the adjoint map.
Let  $(\land(u, v)/\langle u^2, v^2 \rangle,0)$ denote the Sullivan model for $S^p \times
S^q$ where $|u|=p$ and $|v|=q.$
\begin{lemma}  \label{lem:composite} A Sullivan model for the map $F * G$    is the DG algebra  map
$$ \mathcal{A}_{F * G} \colon  \land V \otimes \land W  \to \land V \otimes \land W \otimes \land(u, v)/\langle u^2, v^2 \rangle$$ given by
$$  \mathcal{A}_{F * G}(\chi) = \chi + u   \theta_a(\chi) + v   \theta_b(\chi) + uv \theta_a \circ \theta_b (\chi)  $$  for $ \chi \in \land V \otimes \land W.$
\end{lemma}
\begin{proof}
The map $F * G$ is the composite
$$ \xymatrix{ E \times S^p \times S^q  \ar[rr]^{F \times 1_{S^q}} &&  E \times S^q  \ar[r]^{G} &    E}.$$
By the K\"{u}nneth Theorem, a Sullivan model $\mathcal{A}_{F \times
1_{S^q}}$  for the product $F \times 1_{S^q}$  is the product of the
Sullivan models $\mathcal{A}_F \colon \land V \otimes \land W \to
\land V \otimes \land W \otimes (\land (u)/\langle u^2 \rangle)$ and
$1 \colon \land(v)/\langle v^2 \rangle \to \land(v)/\langle v^2
\rangle$.  Thus, given $\chi \in \land V \otimes \land W$ we see:
$$\begin{array}{ll} \mathcal{A}_{F *G}(\chi) &
= \mathcal{A}_{F \times 1_{S^q}} (\mathcal{A}_G(\chi)) \\ & = \mathcal{A}_{F \times 1_{S^q}}(\chi + v\theta_b(\chi))
\\ &  = \mathcal{A}_F(\chi) +
\mathcal{A}_F(v\theta_b(\chi))  \\
&  =  \chi + u\theta_a(\chi) + v\theta_b(\chi) + uv\theta_a(\theta_b(\chi)).
\end{array}$$
\end{proof}

We can now extend Theorem \ref{thm:main2'} to the fundamental group
  for the monoid $\aut(p)_\circ$.

 \begin{theorem}  Let $p \colon E \to B$ be a fibration of simply connected CW complexes with $E$ finite. Then the map
 $$\Phi' \colon \pi_1(\aut(p)_\circ)   \to H_1(\der_{\land V}(\land V \otimes \land W))$$
is a homomorphism inducing an isomorphism
 $$ \Phi \colon \pi_1(\aut(p)_\circ) \otimes \Q   \stackrel{\cong}{\longrightarrow} H_1(\der_{\land V}(\land V \otimes \land W)).$$
 \end{theorem}
 \begin{proof}
To prove $\Phi'$ is a homomorphism, let $\alpha, \beta \in
\pi_1(\aut(p)_\circ))$ and  recall the basic identity  $$\alpha\cdot
\beta = (\alpha * \beta) \circ \Delta \colon S^1 \to \aut(p)_\circ$$
where $\Delta$ is the diagonal map and the left-hand product is the
usual
 multiplication in the fundamental group.   Thus the adjoint to $\alpha \cdot \beta$
 is the composition
 $$ \xymatrix{S^1 \times E  \ar[rr]^{\Delta \times 1_E} &&  S^1 \times S^1 \times E
 \ar[rr]^{F *G} && E}$$
 The result now follows directly from Lemma \ref{lem:composite} and
 the fact that a Sullivan model for $\Delta$ is the map
 $ \land(u, v) \to  \land(w)$ given by $u,v \mapsto w.$
 The proof that $\Phi$ is a bijection is now the same as in the proof  in Theorem \ref{thm:main2'}.
  \end{proof}

We next prove the map $\Phi' \colon \pi_*(\aut(p)_\circ) \to H_*(\der_{\land V}(\land V \otimes \land W))$  preserves brackets which completes the proof of Theorem \ref{thm:main}.
\begin{theorem}    Let $p \colon E \to B$ be a fibration of simply connected CW complexes with $E$ finite.  Let $\alpha \in \pi_p(\aut(p)_\circ)$ and $\beta \in \pi_q(\aut(p)_\circ).$  Then
$$ \Phi'([\alpha, \beta]) = \left[\Phi'(\alpha), \Phi'(\beta)  \right],$$
where the left-hand bracket is the Samelson product in $\pi_*(\aut(p)_\circ)$
and the right-hand bracket is that  induced by the commutator
 in $ \der^{*}_{\land V}(\land V \otimes \land W).$
\end{theorem}
\begin{proof}
By Proposition \ref{pro:CW2}, $\aut(p)_\circ$ is a grouplike space under composition.
Let $\nu \colon \aut(p)_\circ \to \aut(p)_\circ$ be a homotopy inverse.
Define  $\overline{F}$ to   be the composite
$$ \xymatrix{ \overline{F} \colon E \times S^p \ar[rr]^{1_E \times \alpha} && E \times    \aut(p)_\circ \ar[rr]^{1_E \times \nu} &&  E \times \aut(p)_\circ \ar[r]^{\, \, \, \, \, \, \, \, \, \, \, \, \omega}  & E}$$
where $\omega$ is the evaluation map.  Then $\overline{F}$ is adjoint to $-\alpha = \nu_\sharp(\alpha) \in \pi_p(\aut(p)_\circ)$
and so has Sullivan model
$$\mathcal{A}_{\overline{F}}(\chi)=\chi -u\theta_a(\chi) \colon \land V \otimes \land W \to \land V \otimes \land W \otimes (\land (u)/\langle u^2 \rangle)$$
for $\chi \in \land V \otimes \land W$ where $\langle  \theta_a \rangle = \Phi'(\alpha).$

Now given two classes $\alpha \in \pi_p(\aut(p)_\circ)$ and $\beta \in \pi_q(\aut(p)_\circ),$
 the   Samelson product is defined by means of the map $\gamma\: S^p\times S^q \to \aut(p)_\circ$
defined by $$\gamma(x,y) = \alpha(x)\circ \beta(y)\circ \overline{\alpha(x)} \circ \overline{\beta(y)}$$
where here $\overline{\alpha} = \nu_\sharp( \alpha)$.  The map $\gamma$ has adjoint $\Gamma$ given by the     composite
$$\xymatrix{ E \times S^p \times S^q \ar[rr]^{\! \! \! \! \! \! \! \! \! \! \! \! \! \!
1_E  \times \Delta\times
\Delta } \ar[ddrr]_\Gamma &&    E \times S^p\times S^p \times S^q\times
S^q \ar[d]^{1_E \times 1_{S^p} \times T \times 1_{S^q}} \\
 &&   E \times S^p \times S^q\times S^p\times S^q
\ar[d]^{[F,G]} \\ &&  E}$$ where $T$ is   transposition and   $$[F,G]= F\circ
( G \times 1_{S^n})\circ ( \overline{F} \times 1_{S^p\times S^q})\circ
( \overline{G} \times 1_{S^p\times S^q\times S^p})\,.$$
We see the map $[F, G]$ has Sullivan model
$$ \mathcal{A}_{[F, G]} \colon \land V \otimes \land W \to \land V \otimes \land W \otimes (\land(u,v, \overline{u}, \overline{v})/\langle u^2, v^2, \overline{u}^2, \overline{v}^2 \rangle)$$
given by
$$\begin{array}{ll}\mathcal{A}_{[F, G]}(\chi) & = \chi + u\theta_a(\chi) + v\theta_b(\chi)  - \overline{u}\theta_a(\chi) - \overline{v}\theta_b(\chi) \\

& \, \,  \, \, \, \, \, \,  \, \, \, + \,  \, uv\theta_a \circ \theta_b(\chi)  + \overline{u}\overline{v}\theta_a \circ \theta_b(\chi) -u\overline{v}\theta_a \circ \theta_b(\chi)  -  v\overline{u}\theta_b \circ \theta_a(\chi) \\

& \, \,  \, \, \, \, \, \,  \, \, \,  + \, \hbox{\, terms involving \,}  u\overline{u}  \hbox{\, or } v\overline{v} \end{array}.$$
Thus  $\Gamma$ has Sullivan model
 $$ \mathcal{A}_\Gamma(\chi) = \chi  + uv[\theta_a, \theta_b](\chi)  \, \: \land V \otimes \land W   \to  \land V \otimes \land W  \otimes
 \land(u, v)/\langle u^2, v^2 \rangle$$
for $\chi \in \land V \otimes \land W.$
The result now follows from the definition of the Samelson product:  The restriction of  $\Gamma$  to $E \times (S^p \vee S^q)$ is null and so $\Gamma$   induces
$\Gamma' \colon E \times S^{p+q} \to E$  satisfying $\Gamma \simeq \Gamma' \circ (1_E \times q)$ where
$q \colon S^p \times S^q \to S^{p+q}$ is the projection onto the smash product. Finally, using the fact  that $q$ induces the map $$ w \mapsto uv \, \colon \land(w)/\langle w^2 \rangle \to \land(u, v)/\langle u^2, v^2 \rangle,$$
we see $\Phi'(\gamma)$ is represented by $[\theta_a, \theta_b].$
\end{proof}

\begin{remark}  The techniques above can be applied  to describe a related  space of fibrewise self-homotopy equivalences. Let  $F = p^{-1}(x)$ be the fibre over the basepoint of $p \colon E \to B$ as above.  The
restriction $\mathrm{res}\: \aut(p)  \to \aut(F)$ is a multiplicative
continuous map. Denote the kernel by $\aut^{F}(p)$.  This monoid is
of interest in the study of gauge groups.
 Denote by
 $\overline{\der}_{\land V} (\land V\otimes\land W)$   the
subcomplex of $\der_{\land V}(\land V\otimes \land W)$ consisting
of derivations $\theta $ such that $\theta (V)=0$ and $\theta
(W)\subset \land^+V\otimes \land W$. Then we have a Lie algebra isomorphism
$$\pi_*(\aut^F(p)_\circ)\otimes\mathbb Q\cong H_*(\overline{\der}_{\land
V} (\land V\otimes\land W)).$$
The details of the proof are extensive but entirely similar to the above.
\end{remark}

\section{Applications and Examples}
\label{sec:examples}

We apply \thmref{thm:main} to   expand on some of  the examples mentioned in Section \ref{sec:prelims}.
We begin with the path-space fibration $p \colon   PB
\to B$  as discussed in Example
\ref{ex:Loops B as autxi}. Let $B$ be a simply connected CW complex.
As in  \cite[Sec.16b]{FHT},  $p$ has a relative minimal model of the form $$I \colon (\land V,d) \to (\land V\otimes\land \overline{V},D)$$  where, as usual, $(\land V, d)$ is the Sullivan minimal model for $B$. Here $\overline{V}$ is the desuspension of $V$, $(\overline{V})^n \cong V^{n+1}$.  We recover part of the main result of \cite{nuit}.

\begin{theorem}{\em (F\'{e}lix-Halperin-Thomas)}\label{thm:nuit}
Let $B$ be a simply connected CW complex and  $p\: PB\to B$ the path-space fibration. Then there is an isomorphism of  graded
Lie algebras $$\pi_*(\Omega B)\otimes\mathbb Q \cong H_*(\der_{\land
V} (\land V\otimes \land \overline{V})).$$
\end{theorem}

\begin{proof} The result is a direct consequence of Example \ref{ex:Loops B as autxi} and \thmref{thm:main}.
\end{proof}

We next give a general bound for $\Hnil_\Q(\aut(p)_\circ)$ in terms of the rational homotopy groups of the fibre.
\begin{theorem} Let $p \colon E \to B$ be a fibration of  simply
connected CW complexes with $E$ finite. Let $F = p^{-1}(*)$ be the fibre.  Then
$$  \Hnil_\Q(\aut(p)_\circ) \leq \mathrm{card} \{ n \mid \pi_n(F) \otimes \Q \neq 0 \}.$$
\end{theorem}
\begin{proof} A nontrivial commutator   $[ \theta_1, \ldots, [ \theta_{k-1}, \theta_k]]$
 in $\der_{\land V}^*(\land V \otimes \land W)$
gives rise to  a   sequence $w_{1}, \ldots, w_{k}$    of   vectors in $W$
with degrees $n_1 < \cdots < n_{k}.$ To see this choose $w_{k}$ so that  $[ \theta_1, \ldots, [ \theta_{k-1}, \theta_k]](w_{k}) \neq 0$.  Choose
a nonvanishing summand   $\theta_{i_1} \! \circ \cdots \circ \! \theta_{i_k}(w_{k}) \neq 0.$
Choose $w_{k-1}$ so that that $w_{k-1}$ appears in $\theta_{i_k}(w_k)$ and  $\theta_{i_1} \! \circ   \cdots \circ \! \theta_{i_{k-1}}(w_{k-1}) \neq 0.$ Proceed by induction.
The result then follows from Theorem \ref{thm:main}. \end{proof}

For  a principal $G$-bundle $p \colon E \to B$, we observe that
$\Hnil_\Q(G)$ is often $1$ (e.g., when $G$ is a connected Lie group)
while $$\Hnil_\Q(\aut(G)_\circ) = \mathrm{card} \{ n \mid \pi_n(G) \otimes \Q \neq 0 \}.$$
For   recall $G$ has minimal model $(\land(W), 0)$
with $W \cong \pi_*(G) \otimes \Q$ and so, in this case,  $H_*(\der(\land W); D) \cong \der^*(\land W)$.
Any   sequence $w_{n_1}, \ldots, w_{n_k}$   of vectors in $W$of strictly increasing degree
gives rise to a $k$-length commutator: Set $\theta_i(w_{n_i}) = w_{n_{i-1}}$ and $\theta_i(w_j) = 0$  where the $w_j$ gives some homogeneous basis for $W$ including the $w_{n_i}$.  Thus we see
$$ \Hnil_\Q(G_\circ) \leq \Hnil_\Q(\aut(p)_\circ) \leq \Hnil_\Q(\aut(G)_\circ).$$

Finally, we specialize to a case where we can make a complete calculation of the
rational H-type of $\aut(p)_\circ.$
\begin{theorem} Let $p \colon E \to B$ be
a fibration with fibre $F= S^{2n+1}$ and $E$ and $B$ simply connected CW complexes with
$E$ finite.   Suppose   $i_\sharp  \colon \pi_{2n+1}(S^{2n+1}) \to
\pi_{2n+1}(E)$ is injective.  Then $\aut(p)_\circ$ is rationally homotopy
abelian and, for  each $q\geq 1$, we have isomorphisms
$$\pi_q(\aut(p)_\circ)\otimes\Q \cong H_{2n+1-q}(B;\Q)\,.$$
\end{theorem}

\begin{proof} Let $(\land V,d)\to (\land V\otimes \land (u) ,D)$ be a
relative minimal model for $p$. A derivation $\theta \in \der_{\land
V}(\land V\otimes\land (u))$ is determined by the element $\theta
(u)\in (\land V)^{2n+1-q}$. The derivation $\theta$ is a
cocycle (resp. a coboundary) if $\theta (u)$ is a cocycle (resp. a
coboundary). Further, the commutator bracket of any two such derivations
is directly seen to be trivial. The result thus follows from \thmref{thm:main}.
\end{proof}



\section{On the group of components of $\aut(p)$}%
\label{sec:components}
As mentioned in the introduction, the group  $\mathcal{E}(p) = \pi_0(\aut(p))$ of path components of $\aut(p)$  does not generally localize well.  First
$\mathcal{E}(p)$ is  often non-nilpotent.  Even when $\mathcal{E}(p)$   is nilpotent, it
may not satisfy $\mathcal{E}(p)_\Q = \mathcal{E}(p_\Q)$ --- take $p \colon S^n \to *$, for an easy example.
However, by  work of Dror-Zabrodsky \cite{D-Z} and Maruyama (e.g.,  \cite{Maru}),   various natural  subgroups of $\mathcal{E}(p)$   are nilpotent and do localize well.
We consider  one such example which we denote $\mathcal{E}_\sharp(p)$.
We note that   different  versions of groups of fibrewise equivalences
have also been considered \cite{F-T, F-M, GMPV}.
Our definition of $\mathcal{E}_\sharp(p)$  is chosen to allow for an identification of $(\mathcal{E}_\sharp(p))_\Q$ in the spirit of \thmref{thm:main}.

Suppose, as usual, that $p \colon E \to B$ is a fibration of  simply connected CW complexes with $E$ finite. Then any fibrewise self-map of $E$ is fibrewise homotopic to a based self-map.  Further, any
two based fibrewise self-maps of $E$ that are freely fibrewise homotopic are also based fibrewise
homotopic. Thus in what follows we  work with based fibre-self-equivalences of $p \colon E  \to B.$

Consider the long exact sequence  of the fibration
$$\xymatrix{ \cdots \ar[r] & \pi_{i+1}(B) \ar[r]^{\partial} & \pi_{i}(F) \ar[r]^{j_\#} &Ê \pi_{i}(E) \ar[r]^{p_\#}
&Ê \pi_{i}(B) \ar[r] & \cdots}$$
Each $f \in \aut(p)$ induces an automorphism $f_\#$ of this sequence,
which is the identity on $\pi_*(B)$. Thus $f_\#$ induces an
automorphism of cokernels
$$ \overline{f_\#} : \frac{\pi_*(F)}{\mathrm{im}
\partial} \to \frac{\pi_*(F)}{\mathrm{im}
\partial}.$$
Define $\E_\#(p)$ to be the group of based homotopy classes of
based equivalences in $\aut(p)$ that induce the identity on
$\mathrm{im} \partial \subseteq \pi_*(F)$ through degree equal to
the dimension of $E$, and also induce the identity on the cokernels
above through the dimension of $E$.Ê
Notice that this reduces to the
subgroup of classes that induce the identity on $\pi_*(F)$ through
degree equal to the dimension of $E$ if the fibration is ``Whitehead
trivial," i.e., if the connecting homomorphism $\partial$ is trivial
(through the dimension of $E$) as considered in \cite{F-M}.Ê When $p$ is trivial,  it reduces to the subgroup $\E_\sharp(E)$ of
classes that induce the identity on $\pi_*(E)$ through degree equal
to the dimension of $E$  as in \cite{Maru}.

\begin{theorem} \label{thm: E-sharp localizes} Let $ p\colon E \to B$ be a fibration of simply connected CW complexes with $E$ finite.  Then $\E_\sharp(p)$ is a nilpotent group. Given  any set of primes $\mathbb{P}$ we have $$ \E_\#(p)_\mathbb{P} \cong \E_\#(p_\mathbb{P}).$$
\end{theorem}
\begin{proof}
 Observe $\E_\sharp(p)$ acts nilpotently   on the normal chain
$$0 \vartriangleleft  \mathrm{im}
\partial \vartriangleleft \pi_*(F)$$
through the dimension of $E$. That $\E_\sharp(p)$ is a nilpotent group follows from \cite{D-Z}.Ê
The result on localization is now obtained by adjusting Maruyama's argument \cite{Maru} to the fibrewise setting by replacing the use of a Postnikov decomposition of $E$ with  Moore-Postnikov decomposition of $p \colon E \to B.$  Compare \cite[Th.1.5]{GMPV}.
\end{proof}

By virtue of  \thmref{thm: E-sharp localizes}, we may identify $\E_\#(p)_\Q$ in   terms of certain automorphisms of the Sullivan  model of $E$.  As usual,
 let  $(\land V, d) \to (\land V \otimes \land W, D)$ denote the minimal model of $p$.Ê Let
 $D_0 : W \to V$ be the linear part of $D$, and decompose $W$ as $W = W_0 \oplus
W_1$ with $W_0 = \mathrm{ker} D_0$ and $W_1$ a complement.
Topologically, $W_1$ may be identified with the dual of $\mathrm{im}
\partial = \mathrm{ker} j_\sharp$   and  $W_0$ with
the dual of $\mathrm{coker} \partial = \mathrm{im} j_\sharp$ after rationalization. Write $\aut_\sharp(\land V \otimes \land W)$ for the  group of  automorphisms $\varphi$ of $\land V \otimes W$ which are the identity on  $\land V$ and have
linear part $\varphi_0 : W \to V \oplus W$ satisfying   $(\varphi_0 -1)
(W_0) \subseteq V$ and $(\varphi_0 -1) (W_1) \subseteq V\oplus W_0.$
As in \cite{AL},     we may identify
$\mathcal{E}_\#(p_\Q)$ with the   group of homotopy classes of $\aut_\sharp(\land V \otimes \land W).$

Recall from the introduction that $\der^0_{\#}(\land V\otimes \land W)$ denotes the vector space of
derivations $\theta$ of degree $0$ of $\land V\otimes\land W$ that
satisfy $\theta (V)=0$, $\D(\theta) = 0 $, $\theta (W_0)\subset V
\oplus \land^{\geq 2}(V\oplus W)$, and $\theta (W_1)\subset V \oplus
W_0 \oplus \land^{\geq 2}(V\oplus W)$.Ê  The derivation differential
  defines a linear map
Ê$$ \D : \mbox{Der}^1_{\land V}(\land V\otimes\land W)\to
Ê \mbox{Der}^0_{\#}(\land V\otimes\land W)\,, \hspace{1cm}
Ê \theta \mapsto \D(\theta) = D\theta + \theta D.$$
We are writing   $$H_{0}(\der_\sharp(\land V\otimes\land W)) = \mbox{coker}\,
\D\,.$$

As in \cite[Sec.11]{IHES} and \cite[Pro.12]{Sal}, the correspondence $\theta \mapsto e^\theta$ gives a bijection
   (with inverse $\varphi \mapsto \log(\varphi)$)
   from $\der^0_{\#}(\land V \otimes \land W)$  to $\aut_\sharp(\land V \otimes \land W)$.  Under this
   bijection, compositions of automorphisms $e^\theta\circ e^\phi$
   correspond to ``Baker-Campbell-Hausdorff" products of derivations,
   i.e., $$\log(e^\theta\circ e^\phi) = \theta + \phi + \frac{1}{2}
   [\theta, \phi] + \frac{1}{12}[\theta, [\theta, \phi]] + \cdots.$$

 \begin{theorem} Let $p \: E \to B$ be a fibration of simply connected CW complexes
with $E$ finite.  The assigment $\theta \mapsto e^\theta$ induces an
isomorphism of groups
$$\Psi \colon H_{0}(\der_\sharp(\land V\otimes\land W))  \to
\mathcal{E}_{\#}(p_\Q),$$
where Baker-Campbell-Hausdorff composition of derivations is understood in the
left-hand term.
\end{theorem}

\begin{proof}
 To see that  $\Psi$ is well-defined,       we observe that the assignment $ \theta \mapsto e^{\theta}$ restricts to a  correspondence between boundaries in $\Der^0_\sharp(\land V \otimes \land W)$ and automorphisms
  in $\aut_\sharp(\land V \otimes \land W)$ homotopic to the identity there.
 Suppose
  $\theta = \D(\theta_1)$ with $\theta_1 \in
\mbox{Der}^1_{\land V}(\land V\otimes\land W)$. Then we define a
derivation $\theta_2 $ of $\land V\otimes\land W\otimes \land
(t,dt)$ by $\theta_2(V)=\theta_2(t)=\theta_2(dt) = 0$ and
$\theta_2(w)=t\theta_1(w)$ for $w\in W$. The map $$\mathcal{H} =
e^{\D(\theta_2)} : \land V\otimes\land W\to \land
V\otimes\land W \otimes \land (t,dt)$$ satisfies $p_1\circ \mathcal{H} =
e^{\theta} $ and $p_0\circ \mathcal{H} $ is the identity. Conversely,    a given
homotopy $\mathcal{H} \colon \land V \otimes \land W \to \land V \otimes \land W \otimes \land(t, dt)$ between the identity and $\varphi$  in $\aut_\sharp(\land V \otimes \land W)$ may be chosen to be fibrewise  in  the sense of Proposition \ref{prop:o/u}.  Taking the  $dt$ component of   $\log(\mathcal{H})$ we obtain a derivation $\theta \in \der^1_{\land V}(\land V \otimes \land W)$
with $\D(\theta) = \log(\varphi)$.
The proof that  the induced map $\Psi$ is a bijection now follows the same  line
 as the proof of Theorem \ref{thm:main2}.  \end{proof}

We remark that this  result may be used to analyze the nilpotency
of $\mathcal{E}_{\#}(p_\Q)$, and ultimately, since we have
$\mathcal{E}_{\#}(p_\Q) \cong \mathcal{E}_{\#}(p)_\Q$, that of
$\mathcal{E}_{\#}(p)$, in terms of bracket lengths in
$H_{0}(\der_\#(\land V \otimes \land W))$.

Finally,  observe that the results of this section may also be  applied to give cardinality
results for the order of  $\E(p).$
As a simple example,
let $\pi \colon B \times  S^n \to B$ be a trivial  fibration. If $H^n(B;\Q)=0$ then
${\mathcal E}(\pi)$ is a finite group. On the other hand denote by
$V^n$ the dual  of  the image of the rational
Hurewicz map in $H^n(B; \Q)$.  Then clearly ${\mathcal E}_{\# }(\pi_\Q) \cong H_{0}(\der_\sharp(\land V\otimes\land W)) \cong V^n$. Thus if $V^n$ is nontrivial then  ${\mathcal E}(\pi)$ is infinite in this case.
We conclude with one example involving a nontrivial fibration.

\begin{example}
Let $p \colon
S^7\times S^3 \to S^4$ be the composition $p = \eta\circ p_1$, where
$\eta$ denotes the Hopf map and $p_1$ projection onto the first
factor.  Then the fibre is $S^3 \times S^3$.  A relative minimal model for
$p$ is given by
$$\land(v_4, v_7) \to \land(v_4, v_7)\otimes \land(w_3, w'_3) \to \land(w_3,
w'_3),$$
with (non-minimal) differential given by $D(v_7) = v_4^2$, $D(w'_3)
= v_4$, and $D=0$ on other generators.  With reference to our
notation above, we have $V = \langle v_4, v_7 \rangle$, $W_0 =
\langle w_3 \rangle$, and $W_1 = \langle w'_3 \rangle$.  Define a
differential $\theta \in \Der^0_\#(\land V \otimes \land W)$ by
setting $\theta(w'_3) = w_3$, and $\theta = 0$ on other generators.
A direct check shows that $\theta$ represents a non-zero class in
$H_{0}(\Der_\sharp(\land V \otimes \land W))$.  From Theorem 6.3 and
the discussion of this section, it follows that
$\mathcal{E}_{\#}(p)$ has infinite order.

Under the isomorphism of Theorem 6.3, the derivation $\theta$
evidently corresponds to the automorphism $\varphi$ of $\land V
\otimes \land W$ given by $\varphi(w'_3) = w'_3 + w_3$ and $\varphi
= 1$ on other generators.   Notice that this automorphism does not
correspond to an element of $\mathcal{E}_{\#}(F)$---in the
``non-fibrewise" notation of \cite{D-Z}.  Also, since in this case
the total space has the rational homotopy type of $S^3\times S^7$,
we see that $\varphi$ does not correspond to an element of
$\mathcal{E}_{\#}(E)$, which is trivial rationally. The example thus
demonstrates that, in general,  it is necessary to consider the
relative minimal model of a fibration and not just  the minimal
models of the spaces  involved   to analyze $\mathcal{E}_{\#}(p)$
rationally.
\end{example}


\begin{thebibliography}{[GK]}
\bibliographystyle{amsalpha}

\bibitem{ArkCur} M. Arkowitz and C. R. Curjel, {\em  Homotopy commutators of finite order. II.}  Quart. J. Math. Oxford Ser. (2)  {\bf 15}  (1964), 316--326

\bibitem{AL} M. Arkowitz and G.  Lupton,
{\em On the nilpotency of subgroups of self-homotopy equivalences},  Progr. Math., {\bf 136} (1996),  1--22.


\bibitem{Baues}
H.~J. Baues, \emph{Algebraic homotopy}, Cambridge Studies in
Advanced
  Mathematics, vol.~15, Cambridge University Press, Cambridge, 1989.


\bibitem{B-G}
I.~Berstein and T.~Ganea, {\emph{Homotopical nilpotency}},
Illinois J.\  Math.\  {\textbf{5}} (1961), 99--130.



\bibitem{BL} J. Block and A. Lazarev
{\em Andr\'{e}-Quillen cohomology and rational homotopy of function spaces,} Adv. Math. {\bf 193} (2005), no. 1, 18--39.

\bibitem{BHMP} P. Booth,  P. Heath, C. Morgan, R.  Piccinini
{\em $H$-spaces of self-equivalences of fibrations and bundles,}
Proc. London Math. Soc. (3) {\bf 49}  (1984), no. 1, 111--127.

\bibitem{BM} U. Buijs and A. Murillo, \emph{The rational homotopy Lie algebra of function spaces},  Comment. Math. Helv.  83  (2008),  no. 4, 723--739.
\bibitem{CJ} M. C. Crabb and I. M. James,  {\em Fibrewise homotopy theory}, Springer-Verlag, 1998.




\bibitem{Dold} A. Dold,  {\em Partitions of unity in the theory of fibrations,}  Ann. of Math. (2)  {\bf 78}  (1963) 223--255.
\bibitem{D-Z} E. Dror and A. Zabrodsky
{\em Unipotency and nilpotency in homotopy equivalences,}
Topology {\bf 18} (1979), no. 3, 187--197.

\bibitem{nuit} Y. F\'elix, S. Halperin and J.-C. Thomas,
\emph{Sur certaines alg\'ebres de Lie de d\'erivations}, Annales
Institut Fourier 32 (1982), 143-150.


\bibitem{FHT} \bysame, {\em Rational
homotopy theory}, Springer-Verlag, 2001.

\bibitem{F-M} Y. F\'elix and A. Murillo, {\em A note on the nilpotency of subgroups of self-homotopy equivalences,}
Bull. London Math. Soc. {\bf 29} (1997),486--488.


\bibitem{F-T} Y. F\'elix and  J.-C. Thomas, {\em Nilpotent subgroups of the group of fibre homotopy equivalences,}
Publ. Mat. {\bf 39} (1995), 95--106

  \bibitem{Hal} S. Halperin, {\em Finiteness in the minimal models of
Sullivan,} Trans. Amer. Math. Soc. {\bf 230} (1977), 173-199.







\bibitem{Fuchs}
M. Fuchs, \emph{Verallgemeinerte {H}omotopie-{H}omomorphismen
und
  klassifizierende {R}\"aume}, Math. Ann. \textbf{161} (1965), 197--230.

\bibitem{GMPV}
A. Garv{\'{\i}}n, A. Murillo, P. Pave{\v{s}}i{\'c}
              and A. Viruel, {\em Nilpotency and localization of groups of fibre homotopy
              equivalences}, Contemp. Math. {\bf 274}  (2001), 145-157.



\bibitem{Got}
D.~Gottlieb, \emph{On fibre spaces and the evaluation map}, Ann. Math.
  \textbf{87} (1968), 42--55.

\bibitem{HMR}
P.~Hilton, G.~Mislin, and J.~Roitberg,
{\emph{Localization of nilpotent groups and spaces}},
North-Holland Publishing Co.,
Amsterdam, 1975, North-Holland Mathematics Studies, No.~15,
Notas de Matem\'{a}tica, No.~55 [Notes on Mathematics, No.~55].




\bibitem{PJK} P. J. Kahn,
{\emph{Some function spaces of $CW$ type}}, Proc. Amer. Math. Soc.
{\textbf{90}}, 599-607.






\bibitem{LPSS} G. Lupton, N. C. Phillips, C. L. Schochet and  S. B. Smith,
{\emph{Banach algebras and rational homotopy theory}},
  Trans.\  Amer.\  Math.\  Soc.\ {\bf 361} (2009), 267-295.

\bibitem{LS} G. Lupton and S. B. Smith, \emph{Rationalized evaluation
subgroups of a map I: Sullivan models, derivations and G-sequences},
J. Pure and Applied Algebra {\bf 209} (2007), 159-171.

\bibitem{LS2} \bysame, \emph{Rank of the fundamental
group of any component of a function space}, Proc. Amer. Math. Soc.
{\bf 135} (2007), 2649-2659.

\bibitem{Maru} K. Maruyama, {\em
Localization of a certain subgroup of self-homotopy equivalences},
Pacific J. Math. {\bf 136} (1989), 293--301


\bibitem{Me} W. Meier,
{\em Rational universal fibrations and flag
manifolds,}  Math. Ann. {\bf 258 } (1982),  329-340.

\bibitem{Mil}
J. Milnor,
{\emph{On spaces having the homotopy type of {$CW$}-complex}},
Trans.\  Amer.\  Math.\  Soc.\  {\textbf{90}} (1959), 272-280.

\bibitem{Sal} P. Salvatore, {\em Rational homotopical nilpotency of
self-equivalences,}   Topology and its Appl., {\bf 77} (1997), 37-
50.

\bibitem{Sch}
H.~Scheerer,
{\emph{On rationalized H-~and co-H-spaces. With an appendix on
decomposable H-~and co-H-spaces}},
Manuscripta Math.\  {\textbf{51}} (1985), 63--87.
\MR{88k:55007}

%


\bibitem{IHES} D. Sullivan, \emph{Infinitesimal comutations in
topology}, Publ. I.H.E.S. {\bf 47} (1977), 269-331.


\bibitem{GW}
G. Whitehead, \emph{Elements of homotopy theory}, Graduate Texts in Mathematics,
 Springer-Verlag, New York, 1978.



\end{thebibliography}
\end{document}